\documentclass[11pt,reqno]{amsart}
\usepackage{amsfonts}
\usepackage{amssymb}
\usepackage{a4wide}
\usepackage{pgf,pgfarrows,pgfautomata,pgfheaps,pgfnodes,pgfshade}
\usepackage{url}
\usepackage{lineno}
\usepackage[numbers]{natbib}

\newcommand*\patchAmsMathEnvironmentForLineno[1]{%

  \expandafter\let\csname old#1\expandafter\endcsname\csname #1\endcsname
  \expandafter\let\csname oldend#1\expandafter\endcsname\csname end#1\endcsname
  \renewenvironment{#1}%
     {\linenomath\csname old#1\endcsname}%
     {\csname oldend#1\endcsname\endlinenomath}}%
\newcommand*\patchBothAmsMathEnvironmentsForLineno[1]{%
  \patchAmsMathEnvironmentForLineno{#1}%
  \patchAmsMathEnvironmentForLineno{#1*}}%
\AtBeginDocument{%
\patchBothAmsMathEnvironmentsForLineno{equation}%
\patchBothAmsMathEnvironmentsForLineno{align}%
\patchBothAmsMathEnvironmentsForLineno{flalign}%
\patchBothAmsMathEnvironmentsForLineno{alignat}%
\patchBothAmsMathEnvironmentsForLineno{gather}%
\patchBothAmsMathEnvironmentsForLineno{multline}%
}

\usepackage{hyperref}
\numberwithin{equation}{section}

\newtheorem{defi}{Definition}[section]
\newtheorem{thm}[defi]{Theorem}

\newtheorem{rem}[defi]{Remark}
\newtheorem{cor}[defi]{Corollary}

\newtheorem{problem}[defi]{Problem}

{%
\setcounter{enumi}{0}

\begin{enumerate}}%
{\end{enumerate} }

\newtheorem{asum}{Assumption}

\DeclareMathOperator{\essinf}{ess \text{ }  inf}
\DeclareMathOperator{\esssup}{ess \text{ }  sup}

\newcommand{\diff}{\,\mathrm{d}}
\newcommand{\diffns}{\mathrm{d}}

\begin{document}

\title[Maximum Principle for Markov Switching FBSDEG]{A Maximum Principle for Markov Regime-Switching Forward-Backward Stochastic Differential Games and Applications}


\author{Olivier Menoukeu-Pamen }
\address{ Institute for Financial and Actuarial Mathematics, Department of Mathematical Sciences, University of Liverpool, Peach Street, L69 7ZL, United Kingdom.}
             \email{Menoukeu@liverpool.ac.uk}

             \author{Romuald Herv\'e Momeya }
 \address{ CIBC Asset Management Inc., 1000 de la Gaucheti\`ere Ouest, Montr\'eal, Qu\'ebec,
 Canada.}
              \email{momeya2008@gmail.com}

\thanks{}

\date{ October 2014}



\maketitle


 \begin{abstract}

 In this paper, we present an optimal control problem for stochastic differential
games under Markov regime-switching forward-backward stochastic differential equations with jumps. First, we prove a sufficient maximum principle for non zero-sum
stochastic differential game problems and obtain equilibrium point for such games. Second, we prove an equivalent maximum principle for non zero-sum stochastic differential games. The zero-sum stochastic differential games equivalent maximum principle is then obtained as a corollary. We apply the obtained results to study a problem of robust utility maximization under a relative entropy penalty. We also apply the results to find optimal investment of an insurance firm under model uncertainty.
 \end{abstract}

\section{Introduction}

The expected  utility theory can be seen as the theory of decision making under uncertainty based on some postulates of agent's preferences. In general, the agent's preference is driven by a time-additive functional and a constant rate discount future reward. The standard expected utility maximization problem supposes that the agent knows the initial probability measure that governs the dynamic of the underlying. However, it is difficult or even impossible to find an individual worthwhile probability distribution of the uncertainty. Moreover, in finance and insurance, there is no conformism on which original probability should be used to model uncertainty. This led to the study of utility maximization under model uncertainty, the uncertainty here, being represented by a family of absolute continuous (or equivalent) probability distribution. The idea is to solve the problem for each distribution belonging to the above mentioned class and choose the one that gives the worst objective value. More specifically, the investor maximizes the expected utility with respect to each measure in this class, and chooses among all, the portfolio with the lowest value. This is also called a robust optimization problem  and has been intensively studied the past years. For more information, the reader may consult \cite{BMS07, ElSi2011, FMM10, JMN10, Menou20141, OS111} and references therein.

Stochastic control problem for Markovian regime-switching model has received a lot of attention recently; See e.g., \cite{Donel2011, DonHeu2011, LZ13, Menou20142, TW12, ZES2012}. Each state of the Markov chain represents a state of an economy. Hence, one can include structural changes in economic conditions of the state process. In this paper, we study an optimal control problem of recursive utility for Markov regime-switching jump-diffusion process under model uncertainty. Let mention that the notion of recursive utility was introduced in discrete time in \cite{EZ89, Wei90}, to untie the concepts of risk aversion and intertemporal substitution aversion which are not treated independently in the standard utility formulation. This concept was extended to continuous time in \cite{DE92} and called  stochastic differential utility (SDU). The performance functional in our stochastic differential utility case can be represented as the solution of controlled Markov-switching backward stochastic differential equation (BSDE). As pointed above, the agent seeks the strategy which maximizes the value functional in the worst possible choice of probability distribution. In fact, it is assumed that the mean relative growth rate of the risky asset is not known to the agent, but subject to uncertainty, hence it can be regarded as a stochastic control which plays against the agent. This problem can be seeing as a Markov switching (zero-sum) stochastic differential game between the agent and the market. Such a problem was studied in \cite{ElSi2011}, where the authors introduced a model to discuss an optimal investment problem of an insurance company using a game theoretic approach. The objective of the insurance company being to choose an optimal investment strategy so as to maximize the expected exponential utility of terminal wealth in the worst-case scenario. Their model is general enough to include financial risk, economic risk, insurance risk, and model risk. The stock prices dynamics, the interest rate and the aggregate insurance claim process are modulated by a Markov chain. The authors used the dynamic programming approach to solve the problem and derive explicit solutions. In this paper, we instead use an approach based on stochastic maximum principle, and generalize their results to the framework of (nonzero-sum) forward-backward stochastic differential games and also more general dynamics for the state process. We also obtain explicit formulas for the optimal strategies of the market and the insurance company, when the utility is of  exponential type and the Markov chain has two states. It is worth mentioning that, unlike in \cite{ElSi2011}, in our derivation of the closed forms solutions, we do not assume that the interest rate is zero.

Our paper is also motivated by the idea developed in \cite{Menou20141, Menou20142, OS111}, where the authors derive a general maximum principle for forward-backward stochastic differential games, stochastic differential games with delay and Markov regime-switching stochastic control with partial information, respectively. One important advantage of our approach is that we may relax the assumption of concavity on our Hamiltonian. We derive a general maximum principle for forward-backward Markov regime-switching stochastic differential under model uncertainty. Using this result, we study a problem of recursive utility maximization with entropy penalty. We show that the optimal solution is the unique solution to a quadratic Markov switching backward stochastic differential equation. This result extend the results in  \cite{BMS07, JMN10} by considering a Markov regime-switching state process, and more general stochastic differential utility.

The remaining of the paper is organized as follows: In Section \ref{framew}, we formulate our control problem. In Section \ref{maxiprinc1}, we derive a partial information stochastic maximum principle for forward backward stochastic differential game for a Markov switching L\'evy process under model uncertainty. In Section \ref{application}, we apply our results to study first a robust utility maximization with entropy penalty and second a problem of optimal investment of an insurance company under model uncertainty. In the latter case, explicit expressions for optimal strategies are derived.

\section{Model and Problem Formulation}\label{framew}

In this section, we formulate the general problem of stochastic differential games of Markov regime-switching forward-backward SDEs. Let $(\Omega,\mathcal{F},P)$ be a complete probability space,
where $P$ is a reference probability measure.

Let $\alpha:=\{\alpha(t)\}_{0\leq t\leq T}$ be an irreducible homogeneous continuous-time Markov chain with a finite state space $\mathbb{S}=\{e_1,e_2, \ldots ,e_D\}\subset \mathbb{R}^D$, where $D\in \mathbb{N}$, and the $j$th component of $e_n$ is the Kronecker delta $\delta_{nj}$ for each $n,j=1,\ldots, D$.  Denote by $\Lambda:=\{\lambda_{nj}:1\leq n,j\leq D\}$ the rate (or intensity) matrix of the Markov chain under $P$. Hence, for each $1\leq n,j\leq D,\,\,\lambda_{nj}$ is the transition intensity of the chain from state $e_n$ to state $e_j$ at time $t$. Recall that for $n\neq j,\,\,\lambda_{nj}\geq 0$ and $\sum_{j=1}^D \lambda_{nj}=0$, hence $\lambda_{nn}\leq 0$. It was shown in \cite{EAM94} that $\alpha$ admits the following semimartingale representation
\begin{align}\label{semi-mar1}
\alpha(t)=\alpha(0)+\int_0^t\Lambda ^T\alpha(s)\diffns s+M(t),
\end{align}
where $M:=\{M(t)\}_{t\in [0,T]}$ is a $\mathbb{R}^D$-valued martingale under the measure $P$ with respect to the filtration generated by $\alpha$ and $\Lambda ^T$ denotes the transpose of a matrix. For each $1\leq n,j\leq D$, with $n\neq j$, and $t\in [0,T]$, denote by $J^{nj}(t)$ the number of jumps from state $e_n$ to state $e_j$ up to time $t$. It can be shown (see \cite{EAM94}) that
\begin{align}\label{semi-mar2}
J^{nj}(t)=\lambda_{nj}\int_0^t\langle \alpha(s-),e_n\rangle \diffns s +m_{nj}(t),
\end{align}
where $m_{nj}:=\{m_{nj}(t)\}_{t\in [0,T]}$ with $m_{nj}(t):=\int_0^t\langle \alpha(s-),e_n\rangle \langle \diffns M(s),e_j\rangle$ is a martingale under the measure $P$ with respect to the filtration generated by $\alpha$.

Fix $j\in \{1,2,\ldots,D\}$, denote by $\Phi_j(t)$ the number of jumps into state $e_j$ up to time $t$. Then
\begin{align}\label{semi-mar3}
\Phi_j(t)&:=\sum_{n=1,n\neq j}^D J^{nj}(t)= \sum_{n=1,n\neq j}^D \lambda_{nj}\int_0^t\langle \alpha(s-),e_n\rangle \diffns s +\widetilde{\Phi}_{j}(t)\notag\\
&= \lambda_j(t) + \widetilde{\Phi}_{j}(t),
\end{align}
with $\widetilde{\Phi}_{j}(t)=\sum_{n=1,n\neq j}^D m_{nj}(t)$ and $\lambda_j(t)=\sum_{n=1,n\neq j}^D \lambda_{nj}\int_0^t\langle \alpha(s-),e_n\rangle \diffns s
$. Note that, for each
 $j\in \{1,2,\ldots,D\},\,\,\,\widetilde{\Phi}_{j}:=\{\widetilde{\Phi}_{j}(t)\}_{t\in [0,T]}$ is a martingale under the measure $P$ with respect to the filtration generated by $\alpha$.

Let  $B=\{B(t)\}_{0\leq t\leq T}$ be a Brownian motion and $\widetilde{N}_\alpha(\diffns t, \diffns z):=N(\diffns \zeta,\diffns s)-\nu_\alpha (\diffns \zeta)\diff s$ be an independent compensated Markov regime-switching Poisson random measure with $\nu_\alpha (\diffns \zeta)\diff s$, the compensator (or dual predictable projection) of $N$, defined by:
\begin{align}\label{comp1}
\nu_\alpha(\diffns \zeta)\diffns t:=\sum_{j=1}^D\langle \alpha(t-),e_j\rangle \nu_j(\diffns \zeta)\diffns t.
\end{align}
For each $j\in \{1,2,\ldots,D\},\,\,\nu_j(\diffns \zeta)$ is the conditional density of the jump size when the Markov chain $\alpha$ is in state $e_j$ and satisfies $\int_{\mathbb{R}_{0}}\min(1,\zeta^{2})\nu_j (\diffns \zeta)< \infty$,  where $\mathbb{R}_0=\mathbb{R} \backslash \{0\}$.

Suppose that the state process $X(t)=X^{(u)}(t,\omega);\,\,0 \leq t \leq T,\,\omega \in \Omega$, is a controlled Markov regime-switching jump-diffusion process of the form
\begin{equation}\left\{
\begin{array}{llll}
\diff X (t) & = & b(t,X(t),\alpha(t),u(t),\omega)\diff t +\sigma (t,X(t),\alpha(t),u(t),\omega)\diff B(t)  \\
&  & +\displaystyle \int_{\mathbb{R}_0}\gamma(t,X(t),\alpha(t),u(t),\zeta, \omega)\,\widetilde{N}_\alpha(\diffns \zeta,\diffns t)\\
& &+\eta (t,X(t),\alpha(t),u(t),\omega)\cdot \diffns \widetilde{\Phi}(t),\,\,\,\,\,\, t \in [ 0,T] , \label{eqstateprocess1} \\
X(0) &=& x_0.
\end{array}\right.
\end{equation}

We suppose that $\mathbb{F}=\{\mathcal{F}_t\}_{0\leq t\leq T}$ is the $P$-augmentation of the natural filtration associated with $B$, $N$ and $\alpha$. In our model, $u=(u_1,u_2)$, with $u_i$ being the control of player $i;\, i=1,2.$ We suppose that the different levels of information available at time $t$ to the player $i;\, i=1,2$ are modelled by two subfiltrations
\begin{align}
 \mathcal{E}^{(i)}_t\subset \mathcal{F}_t\,;\,\,\,t\in [0,T].
\end{align}
Denote by $\mathcal{A}_i$ the set of admissible control of player $i$, contained in the set of
$\mathcal{E}^{(i)}_t$-predictable processes; $i=1,2$, with value in $\mathbb{A}_i\subset \mathbb{R}$.

The functions $b,\sigma, \gamma$ and $\eta$ are given such that for all $t,\,\,\,b(t,x,e_n,u,\cdot)$, $ \sigma(t,x ,e_n,u,\cdot)$, $\gamma(t,x,e_n,u,\zeta,\cdot)$ and $\eta(t,x,e_n,u,\cdot),\,\,n=1,\ldots,D$ are $\mathcal{F}_t$- progressively measurable for all $x \in\mathbb{R},\,\,\,u\in \mathbb{A}_1\times \mathbb{A}_2$ and $\zeta \in\mathbb{R}_0$ and \eqref{eqstateprocess1} has a unique strong solution for any admissible control $u\in \mathbb{A}_1\times \mathbb{A}_2$ .

We consider the associated BSDE's in the unknowns  $\Big(Y_i(t), Z_i(t), K_i(t,\zeta), V_i(t)\Big)$ of the form
\begin{equation}\left\{
\begin{array}{llll}
\diff Y_i (t) & = & - g_i(t,X(t),\alpha(t),Y_i(t), Z_i(t), K_i(t,\cdot),V_i(t),u(t))\diff t \,+\,Z_i(t)\diff B(t)  \\
&  & + \displaystyle \int_{\mathbb{R}_0}K_i(t,\zeta)\,\widetilde{N}_\alpha(\diffns \zeta,\diffns t)+V_i(t)\cdot \diffns \widetilde{\Phi}(t);\,\,\, t \in [ 0,T] ,  \label{eqassBSDE} \\
Y_i(T) &=& h_i(X(T),\alpha(T))\,; i=1,2.
\end{array}\right.
\end{equation}
Here $g_i:[0,T]\times \mathbb{R} \times \mathbb{S} \times \mathbb{R} \times \mathbb{R} \times \mathcal{R} \times \mathbb{R} \times \mathbb{A}_1\times \mathbb{A}_2  \rightarrow \mathbb{R}$ and $h:\mathbb{R}\times \mathbb{S}  \rightarrow \mathbb{R}$ are such that the BSDE \eqref{eqassBSDE} has a unique solution for any admissible control $u \in \mathbb{A}_1\times \mathbb{A}_2$. For sufficient conditions for existence and uniqueness of Markov regime-switching BSDEs, we refer the reader to \cite{CoEl10} or \cite{Crep} and references therein.

Let $f_i:[0,T]\times \mathbb{R} \times \mathbb{S}  \times \mathbb{A}_1\times \mathbb{A}_2 \rightarrow \mathbb{R}, \,\,\,\varphi_i:\mathbb{R}\times \mathbb{S}  \rightarrow \mathbb{R}$ and $\psi_i:\mathbb{R}  \rightarrow \mathbb{R},\,i=1,2$ be given $C^1$ functions with respect to their arguments and $\psi_i^\prime(x)\geq 0,$ for all $x,\,i=1,2$. Assume that the performance functional for each player $i,\,\,i=1,2$ is as follows
\begin{align}
J_i(t,u):=E\Big[ \int_t^T  f_i(s,X(s),\alpha(s),u(s))\diff s + \varphi_i(X(T),\alpha(T))\,+\,\psi_i(Y_i(t))\Big| \mathcal{E}^{(i)}_t\Big]; i=1,2.\label{perfunct1}
\end{align}
Here, $f_i,\,\varphi_i$ and $\psi_i$ may be seen as profit rates, bequest functions and ``utility evaluations" respectively, of the player $i;\, i=1,2$. For $t=0$, we put
 \begin{align}
J_i(u):=J_i(0,u),\,\, i=1,2.\label{perfunct2}
\end{align}
 We shall consider the non-zero sum stochastic differential game problem, that is, we analyze the following:
 \begin{problem}\label{prob1}
 Find $(u_1^\ast,u_2^\ast)\in \mathcal{A}_1\times\mathcal{A}_2$ (if it exists) such that
 \begin{enumerate}
 \item $J_1(t,u_1,u_2^\ast)\leq J_1(t,u_1^\ast,u_2^\ast)$ for all $u_1\in \mathcal{A}_1$,
 \item $J_2(t,u_1^\ast,u_2)\leq J_2(t,u_1^\ast,u_2^\ast)$ for all $u_2\in \mathcal{A}_2$.
 \end{enumerate}
 \end{problem}
If it exists, we call such a pair $(u_1^\ast,u_2^\ast)$ a Nash Equilibrium. This intuitively means that while player I controls $u_1$, player II controls $u_2$. We assume that each player knows the equilibrium strategies of the other player and does not gain anything by changing his strategy unilaterally. If each player is making the best decision she can, based on the other player's decision, then we say that the two players are in Nash Equilibrium.

\section{A stochastic maximum principle for Markov regime-switching forward-backward stochastic differential games}\label{maxiprinc1}

In this section, we shall find the Nash equilibrium for Problem \ref{prob1} based on a stochastic maximum principle for Markov regime-switching forward-backward differential equation.

Define the Hamiltonians
\begin{equation*}
    H_i:[0,T] \times \mathbb{R}\times \mathbb{S}\times \mathbb{R}\times \mathbb{R}\times \mathcal{R}\times \mathbb{R} \times \mathbb{A}_1\times \mathbb{A}_2 \times\mathbb{R}  \times \mathbb{R}\times \mathbb{R} \times \mathcal{R} \times \mathbb{R}
    \longrightarrow \mathbb{R},
\end{equation*}
by
\begin{align}
   & H_i\left(t,x,e_n,y,z,k,v,u_1,u_2,a,p,q,r(\cdot),w\right)\notag\\
    :=& f_i(t,x,e_n,u_1,u_2)+a g_i(t,x,e_n,y,z,k,v,u_1,u_2)+  p_i b(t,x,e_n,u_1,u_2) \,\notag\\
    &+q_i\sigma (t,x,e_n,u_1,u_2)+\int_{\mathbb{R}_0}r_i(\zeta)\gamma(t,x,e_n,u_1,u_2,\zeta)\nu_\alpha(\diffns \zeta)\notag\\
    & +\sum_{j=1}^D\eta_j(t,x,e_n,u_1,u_2)w_n^j(t)\lambda_{nj} ,\,\,\,i=1,2\label{Hamiltoniansddg1}
\end{align}
where $\mathcal{R} $ denote the set of all functions $k:[0,T]\times \mathbb{R}_0 \rightarrow \mathbb{R}$ for which the integral in \eqref{Hamiltoniansddg1} converges. An example of such set is the set $L^2(\nu_\alpha)$. We suppose that $H_i,\,i=1,2$ is Fr\'echet differentiable in the variables $x, y, z, k, v, u$ and that $\nabla_k H_i(t,\zeta),\,i=1,2$ is a random measure which is absolutely continuous with respect to $\nu$. Next, we define the associated adjoint process $A_i(t),\,p_i(t),\,q_i(t), r_i(t,\cdot)$ and $w_i(t),\,\,\,t\in[0,T]$ and $\zeta\in\mathbb{R}$ by the following forward-backward SDE

 \begin{enumerate}
 \item The Markovian regime-switching forward SDE in $A_i(t); \,i=1,2$
 \begin{equation}\left\{
 \begin{array}{llll}
\diffns  A_i (t) & = & \dfrac{\partial H_i}{\partial y} (t) \diff t + \dfrac{\partial H_i}{\partial z} (t) \diffns B(t)+ \displaystyle \int_{\mathbb{R}_0} \dfrac{\diffns \nabla_k H_i}{\diffns \nu(\zeta)} (t,\zeta) \,\widetilde{N}_\alpha(\diffns \zeta,\diffns t)\\
 && + \nabla_vH_i(t)\cdot\diffns \widetilde{\Phi}(t);\,\,\,t\in [0,T], \\
 A_i(0) &=& \psi_i^\prime(Y(0)).\label{lambda1}
 \end{array}\right.
 \end{equation}
 Here and in what follows, we use the notation
 $$
  \dfrac{\partial H_i}{\partial y} (t) = \dfrac{\partial H_i}{\partial y} (t,X(t),\alpha(t),u_1(t),u_2(t),Y_i(t), Z_i(t), K_i(t,\cdot),V_i(t),A_i(t),p_i(t),q_i(t),r_i(t,\cdot),w_i(t)),
 $$
etc, $\dfrac{\diffns \nabla_k H_i}{\diffns \nu(\zeta)} (t,\zeta) $ is the Radon-Nikodym derivative of $ \nabla_k H_i(t,\zeta)$ with respect to $\nu(\zeta)$ and $\nabla_v H_i(t)\cdot\diffns \widetilde{\Phi  }(t)=\sum_{j=1}^D \dfrac{\partial H_i}{\partial v^j} (t)\diffns \widetilde{\Phi}_j(t)$ with $V_i^j=V_i(t,e_j)$.
 \item The Markovian regime-switching BSDE in $(p_i(t),q_i(t),r_i(t,\cdot),w_i(t)); \,i=1,2$
\begin{equation}\left\{
\begin{array}{llll}
\diffns p_i (t) & = & - \dfrac{\partial H_i}{\partial x} (t) \diffns t + q_i(t)\diff B(t)+ \displaystyle  \int_{\mathbb{R}_0} r_i (t,\zeta) \,\widetilde{N}_\alpha(\diffns\zeta,\diffns t)\\
&&+ w_i(t)\cdot\diffns \widetilde{\Phi_i}(t);\,\,\,t\in [0,T],  \\
p_i (T) &=& \dfrac{\partial \varphi_i}{\partial x}(X(T),\alpha(T))\,+A_i(T) \dfrac{\partial h_i}{\partial x} (X(T),\alpha(T))\label{ABSDE1}.
\end{array}\right.
\end{equation}
\end{enumerate}

\subsection{A sufficient maximum principle}

In what follows, we give the sufficient maximum principle.

\begin{thm}[Sufficient maximum principle for Regime-switching FBSDE games]\label{mainressuf1}
\text{} Let $(\widehat{u}_1,\widehat{u}_2)\in \mathcal{A}_1\times\mathcal{A}_2 $  with corresponding solutions $\widehat{X}(t),(\widehat{Y}_i(t), \widehat{Z}_i(t), \widehat{K}_i(t,\zeta), \widehat{V}_i(t)), \widehat{A}_i(t),(\widehat{p}_i(t),\widehat{q}_i(t),\widehat{r}_i(t,\zeta),\widehat{w}_i(t))$ of
\textup{\eqref{eqstateprocess1}}, \textup{\eqref{eqassBSDE}}, \textup{\eqref{lambda1}} and \textup{\eqref{ABSDE1}} respectively for $i=1,2$. Suppose that the following holds:
\begin{enumerate}
\item The functions
\begin{align}\label{thsuffcond1}
x \mapsto  h_i(x, e_n),\,\,x \mapsto  \varphi_i(x, e_n),\,\,y \mapsto  \psi_i(y),
\end{align}
are concave for $i=1,2$.

\item  The functions
\begin{align}\label{thsuffcond21}
 &\widetilde{\mathcal{H}}_1(t,x,e_n,y,z,k,v)\notag\\
 =&\underset{\mu_1\in \mathcal{A}_1 }{\esssup} E\Big[ H_1( t,x,\mu_1,e_n,y,z,k,v,\mu_1,\widehat{u}_2(t),\widehat{A}_1,\widehat{p}_1(t),\widehat {q}_1(t),\widehat{r}_1(t,\cdot ),\widehat{w}_1(t))\Big| \mathcal{E}^{(1)}_t\Big ]
\end{align} and
\begin{align}\label{thsuffcond22}
&\widetilde{\mathcal{H}}_2(t,x,e_n,y,z,k,v)\notag\\
=&\underset{\mu_2\in \mathcal{A}_2 }{\esssup} E\Big [ H_2( t,x,\mu_1,e_n,y,z,k,v,\widehat{u}_1(t),\mu_2,\widehat{A}_2,\widehat{p}_2(t),\widehat{q}_2(t),\widehat{r}_2(t,\cdot),\widehat{w}_2(t))\Big| \mathcal{E}^{(2)}_t\Big ]
\end{align} are all concave for all $(t,e_n) \in  [0,T]\times \mathbb{S}$ a.s.

\item \begin{align}\label{thsuffcond4}
E\Big[\hat{H}_1(t,\widehat{u}_1(t),\widehat{u}_2(t))) \Big. \Big |  \mathcal{E}^{(1)}_t\Big]=\underset{\mu_1\in \mathcal{A}_1 } {\esssup}&\Big\{E\Big[\hat{H}_1(t,\mu_1,\widehat{u}_2(t)) \Big. \Big |  \mathcal{E}^{(1)}_t\Big]\Big\} \end{align}
for all $t\in [0,T]$, a.s.\\
 and
\begin{align}\label{thsuffcond4}
E\Big[ \hat{H}_2(t,\widehat{u}_1(t),\widehat{u}_2(t)) \Big. \Big |  \mathcal{E}^{(2)}_t\Big ]=\underset{\mu_2\in \mathcal{A}_2 } {\esssup}\Big\{E\Big[\hat{H}_2(t,\widehat{u}_1(t),\mu_2(t)) \Big. \Big |  \mathcal{E}^{(2)}_t\Big ]\Big\}
\end{align}
for all $t\in [0,T]$, a.s. Here

\begin{align*}
&\hat{H}_i(t,u_1,u_2(t))\\
=&
H_i(t,\widehat{X}(t),\alpha(t),\widehat{Y}_i(t), \widehat{Z}_i(t), \widehat{K}_i(t,\cdot),\widehat{V}_i(t),u_1(t),u_2(t),\widehat{A}_i(t),\widehat{p}_i(t),
\widehat{q}_i(t),\widehat{r}_i(t,\cdot),\widehat{w}_i(t))
\end{align*}
for $i=1,2.$
\item  $\frac{\diffns }{\diffns \nu}\nabla_k\widehat{g}_i(t,\xi)>-1$ for $i=1,2.$
\item In addition, the integrability condition
\begin{align}\label{thsuffcond5}
&E\Big[\int_0^T\Big\{  \widehat{p}_i^2(t)  \Big(  (\sigma(t)-\widehat{\sigma}(t))^2+ {\int}_{\mathbb{R}_0}( \gamma(t,\zeta)-\widehat{\gamma} (t,\zeta) )^2\,\nu_\alpha(\diffns \zeta)+\sum_{j=1}^D(\eta_j(t)-\widehat{\eta}_j(t) )^2\lambda_{j}(t) \Big)\Big.    \Big.\notag\\
&+(X(t)-\widehat{X}(t))^2 \Big(  \widehat{q}_i^2(t)+  {\int}_{\mathbb{R}_0}\widehat{r}_i^2 (t,\zeta) \nu_\alpha(\diffns  \zeta)+\sum_{j=1}^D(w^j_i)^2(t)\lambda_{j}(t)  \Big)\notag\\
&+(Y_i(t)-\widehat{Y}_i(t))^2 \Big( (\dfrac{\partial \widehat{H}_i} {\partial z} )^2(t)
+  {\int}_{\mathbb{R}_0} \Big\| \nabla_k \widehat{H}_i(t,\zeta)\Big\|^2 \nu_\alpha(\diffns \zeta)+ \sum_{j=1}^D  (\dfrac{\partial \widehat{H}_i} {\partial v^j} )^2(t) \lambda_{j}(t)  \Big)\notag\\
&\Big.    \Big. +\widehat{A}_i^2(t)  \Big(  (Z_i(t)-\widehat{Z}_i(t))^2+ {\int}_{\mathbb{R}_0}( K_i (t,\zeta)-\widehat{K}_i (t,\zeta) )^2\nu_\alpha(\diffns  \zeta) + \sum_{j=1}^D(V_i^j(t)-\widehat{V}_i^j(t) )^2\lambda_{j}(t)   \Big)\Big\} \diffns t \Big]<\infty
\end{align}
for $i=1,2.$ holds.
\end{enumerate}
Then $\widehat{u}=(\widehat{u}_1(t),\widehat{u}_2(t))$ is a Nash equilibrium for \textup{\eqref{eqstateprocess1}}, \textup{\eqref{eqassBSDE}} and \textup{\eqref{perfunct1}}.
\end{thm}
\begin{rem}
In the above Theorem and in the proof in Appendix, we have used the following shorthand
notation: For $ i = 1$, the processes corresponding to $u=( u_1,\hat{u}_2)$ are given for example by $X(t) = X^{(u_1,\hat{u}_2)}(t)$ and $Y_1(t) = Y_1^{(u_1,\hat{u}_2)}(t)$ and the processes corresponding to $u= (\hat{u}_1,\hat{u}_2)$ are $\hat{X}(t) = X^{(\hat{u}_1,\hat{u}_2)}(t)$ and $\hat{Y}_1(t) = Y_1^{(\hat{u}_1,\hat{u}_2)}(t) $. Similar notation is used for $i=2$.
\end{rem}
\begin{rem}

Let $V$ be an open subset of a Banach space $\mathcal{X}$ and let $F: V \rightarrow \mathbb{R}$.
\begin{itemize}
\item We say that $F$ has a \textsl{directional derivative} (or Gateaux derivative) at $x\in V$ in the direction $y\in \mathcal{X}$ if
$$
D_yF(x):=\underset{\varepsilon \rightarrow 0}{\lim} \frac{1}{\varepsilon}(F(x + \varepsilon y)-F(x)) \text{ exists.}
$$
\item We say that $F$ is Fr\' echet differentiable at $x \in V$ if there exists a linear map
$$
L:\mathcal{X} \rightarrow \mathbb{R}
$$
such that
$$
\underset{\underset{h \in \mathcal{X}}{h \rightarrow 0}}{\lim} \frac{1}{\|h\|}|F(x+h)-F(x)-L(h)|=0.
$$
In this case we call $L$ the Fr\' echet derivative of $F$ at $x$, and we write
$$
L=\nabla_x F.
$$
\item If $F$ is Fr\' echet differentiable, then $F$ has a directional derivative in all directions
$y \in \mathcal{X}$ and
$$
D_yF(x)= \nabla_x F(y).
$$
\end{itemize}

\end{rem}

\begin{proof}
See Appendix.
\end{proof}

\subsection{An equivalent maximum principle}
The concavity condition on the Hamiltonians does not always hold on many applications. In this section, we shall prove an equivalent stochastic maximum principle which does not require this assumption. We shall assume the following:
\begin{asum}\label{assumces1}
For all $t_0\in [0,T]$ and all bounded $\mathcal{E}^{(i)}_{t_0}$-measurable random variable $\theta_i(\omega)$, the control process $\beta_i$ defined by
\begin{align}\label{eqbeta1}
\beta_i(t):= \chi_{]t_0,T[}(t)\theta_i(\omega);\,\,t\in [0,T],
\end{align}
belongs to $\mathcal{A}_i,\, i=1,2$.
\end{asum}

\begin{asum}\label{assumces2}
For all $u_i \in \mathcal{A}_i$ and all bounded $\beta_i \in \mathcal{A}_i$, there exists $\delta_i>0$ such that
\begin{align}\label{eqbeta1}
\widetilde{u}_i(t):=u_i(t)+\ell \beta_i(t) \,\,t\in [0,T] ,
\end{align}
belongs to $\mathcal{A}_i$ for all $\ell \in ]-\delta_i,\delta_i[,\, i=1,2$.
\end{asum}

\begin{asum}\label{assumces3}
For all bounded $\beta_i \in \mathcal{A}_i$, the derivatives processes
\begin{align*}
X_1(t)=\dfrac{\diffns}{\diffns \ell }X^{(u_1+\ell \beta_1,u_2)}(t)\Big. \Big|_{\ell=0};
X_2(t)=\dfrac{\diffns}{\diffns \ell }X^{(u_1,u_2+\ell \beta_2)}(t)\Big. \Big|_{\ell=0};\\
y_1(t)=\dfrac{\diffns}{\diffns \ell }Y_1^{(u_1+\ell \beta_1,u_2)}(t)\Big. \Big|_{\ell=0};
y_2(t)=\dfrac{\diffns}{\diffns \ell }Y_2^{(u_1,u_2+\ell \beta_2)}(t)\Big. \Big|_{\ell=0};\\
z_1(t)=\dfrac{\diffns}{\diffns \ell }Z_1^{(u_1+\ell \beta_1,u_2)}(t)\Big. \Big|_{\ell=0};
z_2(t)=\dfrac{\diffns}{\diffns \ell }Z_2^{(u_1,u_2+\ell \beta_2)}(t)\Big. \Big|_{\ell=0};\\
k_1(t,\zeta)=\dfrac{\diffns}{\diffns \ell }K_1^{(u_1+\ell \beta_1,u_2)}(t,\zeta)\Big. \Big|_{\ell=0};
k_2(t,\zeta)=\dfrac{\diffns}{\diffns \ell }K_2^{(u_1,u_2+\ell \beta_2)}(t,\zeta)\Big. \Big|_{\ell=0};\\
v_1^j(t)=\dfrac{\diffns}{\diffns \ell }V_1^{j,{(u_1+\ell \beta_1,u_2)}}(t)\Big. \Big|_{\ell=0},\,\,\,j=1,\ldots, n;
v_2^j(t)=\dfrac{\diffns}{\diffns \ell }V_2^{j,{(u_1,u_2+\ell \beta_1)}}(t)\Big. \Big|_{\ell=0},\,\,\,j=1,\ldots, n
\end{align*}
exist and belong to $L^2([0,T] \times \Omega)$.
\end{asum}
It follows from \eqref{eqstateprocess1} and \eqref{eqassBSDE} that

\begin{equation}\left\{
\begin{array}{llll}
\diffns X_1(t) & =&X_1(t) \Big\{\dfrac{\partial b}{\partial x}(t)+\dfrac{\partial \sigma}{\partial x}(t)+ \displaystyle  \int_{\mathbb{R}_0} \dfrac{\partial \gamma}{\partial x}(t,\zeta)\widetilde{N}_\alpha(\diffns t,\diffns \zeta) + \dfrac{\partial \eta}{\partial x}(t)\cdot\diffns \widetilde{\Phi}(t) \Big\}\\
&&+\beta_1(t)\Big\{\dfrac{\partial b}{\partial u_1}(t)+\dfrac{\partial \sigma}{\partial u_1}(t)+ \displaystyle  \int_{\mathbb{R}_0} \dfrac{\partial \gamma}{\partial u_1}(t,\zeta)\widetilde{N}_\alpha(\diffns t,\diffns \zeta) +\dfrac{\partial \eta}{\partial u_1}(t)\cdot \diffns \widetilde{\Phi}(t)  \Big\};\,\,t\in(0,T] \\
X_1(0)&=&0.\label{derivstate1}
\end{array}\right.
\end{equation}
and
\begin{equation}\left\{
\begin{array}{llll}
\diffns y_1(t)&=&-\Big\{\dfrac{\partial g_1}{\partial x}(t)x_1(t)+\dfrac{\partial g_1}{\partial y}(t)y_1(t)+\dfrac{\partial g_1}{\partial z}(t)z_1(t)+\displaystyle  \int_{\mathbb{R}_0}\nabla_k g_1 (t)k_1(t,\zeta)\nu_\alpha(\diffns \zeta)\Big. \\
&&+\sum_{j=1}^D \dfrac{\partial g_1}{\partial v_1^j}(t)v_1^j(t)\lambda_j(t)+\dfrac{\partial g_1}{\partial u}(t)\beta_1(t)\Big\} \diffns t +z_1(t)\diff  B(t)  \\
&&+\displaystyle \int_{\mathbb{R}_0}k_1(t,\zeta) \widetilde{N}_\alpha(\diffns \zeta, \diffns t) + v_1(t)\cdot\diffns \widetilde{\Phi}(t) ;\,\,\,t\in [0,T]\\
y_1(T)&=&\dfrac{\partial h_1}{\partial x}(X(T),\alpha(T))x_1(T) .\label{derivassBSDE1}
\end{array}\right.
\end{equation}
We can obtain $\diffns X_2(t)$ and $\diffns y_2(t)$ in a similar way.

\begin{rem}\label{rem111}
As for sufficient conditions for the existence and uniqueness of solutions \eqref{derivstate1} and \eqref{derivassBSDE1}, the reader may consult \textup{\cite[Eq. (4.1)]{Pen93} }(in the case of diffusion state process).

As an example, a set of sufficient conditions under which \eqref{derivstate1} and \eqref{derivassBSDE1} admit a unique solution is as follows:

\begin{enumerate}
\item Assume that the coefficients $b,\sigma, \gamma, \eta, g_i, h_i,f_i, \psi_i$ and $\phi_i$ for $i=1,2$ are continuous with respect to their arguments and are continuously differentiable  with respect to $(x,y,z,k,v,u)$. (Here, the dependence of $g_i$ and $f_i$ on $k$ is trough $\int_{\mathbb{R}_0}k(\zeta)\rho(t,\zeta)\nu(\diffns \zeta)$, where $\rho$ is a measurable function satisfying $0\leq \rho(t,\zeta)\leq c(1\wedge |\zeta|), \text{  } \forall \zeta\in\mathbb{R}_0$. Hence the differentiability in this argument is in the Fr\'echet sense.)

\item The derivatives of $b,\sigma, \gamma, \eta$ and $h_i,g_i,\,i=1,2$ are bounded.

\item The derivatives of $f_i,\,i=1,2$ are bounded by $C(1+|x|+|u|)$.

\item The derivatives of $\psi_i$ and $\phi_i$ with respect to $x$ are bounded by $C(1+|x|).$
\end{enumerate}

\end{rem}

We can state the following equivalent maximum principle

\begin{thm}[Equivalent Maximum Principle]\label{theomainneccon1}
Let $u_i\in \mathcal{A}_i$ with corresponding solutions $X(t)$ of \textup{\eqref{eqstateprocess1}}, $(Y_i(t),Z_i(t),K_i(t,\zeta),V_i(t))$ of \textup{\eqref{eqassBSDE}}, $A_i(t)$ of \textup{\eqref{lambda1}}, $(p_i(t),q_i(t),r_i(t,\zeta),w_i(t))$ of \textup{\eqref{ABSDE1}} and corresponding derivative processes $X_i(t)$ and $(y_i(t),z_i(t),k_i(t,\zeta),v_i(t))$ given by \textup{\eqref{derivstate1}} and \textup{\eqref{derivassBSDE1}}, respectively. Suppose that
\textup{Assumptions \ref{assumces1}, \ref{assumces2}} and \textup{\ref{assumces3}} hold. Moreover, assume the following integrability conditions

 \begin{align}
&E\Big[\int_0^T p_i^2(t)\Big\{\Big(\dfrac{\partial \sigma}{\partial x}\Big)^2(t)X^2_i(t) +\Big(\dfrac{\partial \sigma}{\partial u_i}\Big)^2(t)\beta_i^2(t) \Big.\notag\\
&\Big.+\int_{\mathbb{R}_0}\Big( \Big(\dfrac{\partial \gamma}{\partial x}\Big)^2(t,\zeta)X_i^2(t)+\Big(\dfrac{\partial \gamma}{\partial u_i}\Big)^2(t,\zeta)\beta_i^2(t)\Big)\nu_\alpha (\diffns \zeta)\Big.\notag\\
&\Big.+\sum_{j=1}^D \Big(\Big(\dfrac{\partial \eta^j}{\partial x}\Big)^2(t)x^2_i(t)+\Big(\dfrac{\partial \eta^j}{\partial u_i}\Big)^2(t)\beta_i^2(t)\Big)\lambda_j(t)\Big\} \diffns t\notag\\
&+\int_0^TX_i^2(t)\Big\{ q_i^2(t)+\int_{\mathbb{R}_0}r_i^2(t,\zeta)\nu_\alpha(\diffns \zeta)+\sum_{j=1}^D  (\eta^j)^2(t)\lambda_j(t)\Big\}\diffns t\Big] <\infty \label{thneccond1}
\end{align}
and
\begin{align}
&E\Big[\int_0^Ty_i^2(t)  \Big\{ (\dfrac{\partial H_i}{\partial z})^2 (t) \,+\,\int_{\mathbb{R}_0} \|\nabla_k H_i\|^2(t,\zeta)\nu_\alpha (\diffns\zeta)+  \sum_{j=1}^D(\dfrac{\partial H_i}{\partial v^j})^2 (t) \lambda_j(t)\Big\}\diffns t\notag\\
&+ \int_0^TA_i^2(t)\Big\{z_i^2(t)+\int_{\mathbb{R}_0}k_i^2(t,\zeta)\nu_\alpha (\diffns \zeta)+ \sum_{j=1}^D (v^j_i)^2(t)\lambda_j(t)\Big\}\diffns t\Big]<\infty \text{ for } i=1,2.
\label{thneccond12}
\end{align}
Then the following are equivalent:

 \textup{(1)} $\dfrac{\diffns}{\diffns \ell}J_1^{(u_1+\ell \beta_1,u_2)}(t)\Big. \Big|_{\ell=0}=\dfrac{\diffns}{\diffns \ell}J_2^{(u_1,u_2+\ell \beta_2)}(t)\Big. \Big|_{\ell=0}=0$
 for all bounded  $\beta_1\in \mathcal{A}_{1},\,\beta_2\in \mathcal{A}_{2}$

\textup{(2)}
 \begin{align}\label{thnecond3}
0=E\Big[\dfrac{\partial H_1}{\partial \mu_1} (t,X(t),\alpha(t),\mu_1,u_2,Y_1(t), Z_1(t), K_1(t,\cdot),V_1(t),\notag&\\
A_1(t),p_1(t),q_1(t),r_1(t,\cdot),w_1(t))\Big. \Big| \mathcal{E}^{(1)}_t\Big]_{\mu_1=u_1(t)}&\notag\\
=E\Big[\dfrac{\partial H_2}{\partial \mu_2} (t,X(t),\alpha(t),u_1,\mu_2,Y_2(t), Z_2(t), K_2(t,\cdot),V_2(t),&\notag\\
A_2(t),p_2(t),q_2(t),r_2(t,\cdot),w_2(t))\Big. \Big| \mathcal{E}^{(2)}_t\Big]_{\mu_2=u_2(t)}&
\end{align}
for a.a. $t \in [0,T].$

\end{thm}

\begin{proof}
See Appendix.
\end{proof}

\begin{rem}
The result is the same if we start from $t\geq 0$ in the performance functional, hence extending \textup{\cite[Theorem 2.2]{OS111}} to the Markov regime-switching setting.
\end{rem}

\subsection*{Zero-sum Game}

In this section, we solve the \textsl{zero-sum} Markov regime-switching forward-backward stochastic differential games problem (or \textsl{worst case scenario optimal problem}): That is, we assume that the performance functional for Player II is the negative of that of Player I, i.e.,

\begin{align}
&J(t,u_1,u_2)=J_1(t,u_1,u_2)\nonumber\\
  &:=E\Big[ \int_t^T  f(s,X(s),\alpha(s),u_1(s),u_2(s))\diff s  + \varphi(X(T),\alpha(T))\,+\,\psi(Y(t))\Big.\Big| \mathcal{F}_t\Big]\nonumber \\
        &=:-J_2(t,u_1,u_2).
  \end{align}

In this case, $(u_1^\ast ,u_2^\ast )$ is a \textsl{Nash equilibrium} iff
\begin{align}\label{eqperfunctionzerosum111}
\underset{u_1 \in \mathcal{A}_1}{\esssup}\,J(t,u_1,u_2^\ast)=J(t,u_1^\ast,u_2^\ast)=\underset{u_2 \in \mathcal{A}_2}{\essinf}\,J(t,u_1^\ast,u_2).
  \end{align}
On one hand, \eqref{eqperfunctionzerosum111} implies that
\begin{align*}
 \underset{u_2 \in \mathcal{A}_2}{\essinf}(\underset{u_1 \in \mathcal{A}_1}{\esssup}\,J(t,u_1,u_2))&\leq \underset{u_1 \in \mathcal{A}_1}{\esssup}\,J(t,u_1,u_2^\ast)\\
 &=J(t,u_1^\ast,u_2^\ast)=\underset{u_2 \in \mathcal{A}_2}{\essinf}\,J(t,u_1^\ast,u_2)\\
 &\leq \underset{u_1 \in \mathcal{A}_1}{\esssup}( \underset{u_2 \in \mathcal{A}_2}{\essinf}\,J(t,u_1,u_2)).
\end{align*}
On the other hand, we always have $\essinf(\esssup) \geq \esssup(\essinf)$. Hence, if
$(u_1^\ast ,u_2^\ast )$ is a \textsl{saddle point}, then $$\underset{u_2 \in \mathcal{A}_2}{\essinf}(\underset{u_1 \in \mathcal{A}_1}{\esssup}\,J(t,u_1,u_2))=\underset{u_1 \in \mathcal{A}_1}{\esssup}( \underset{u_2 \in \mathcal{A}_2}{\essinf}\,J(t,u_1,u_2)).$$
The zero-sum Markov regime-switching forward-backward stochastic differential game problem is therefore the following:
\begin{problem}\label{problemmaxinssdgzerosum}
Find $u_1^\ast \in \mathcal{A}_1$ and $u_2^\ast \in \mathcal{A}_2$ (if it exists) such that
    \begin{align}\label{eqproblemmaxinssdgzerosum1}
     \underset{u_2 \in \mathcal{A}_2}{\essinf}(\underset{u_1 \in \mathcal{A}_1}{\esssup}\,J(t,u_1,u_2))=J(t,u_1^\ast,u_2^\ast)=\underset{u_1 \in \mathcal{A}_1}{\esssup}( \underset{u_2 \in \mathcal{A}_2}{\essinf}\,J(t,u_1,u_2)).
    \end{align}
\end{problem}
When it exists, a control $(u_1^\ast,u_2^\ast)$ satisfying \eqref{eqproblemmaxinssdgzerosum1}, is called a \textsl{saddle point}. The actions of the players are opposite, more precisely, between player I and II there is a payoff $J(t,u_1, u_2)$ and it is a reward for Player I and cost for Player II.
\begin{rem}
 As in the non-zero sum case, we shall give the result for $t=0$ and get the result for $t\in]0,T]$ as a corollary. The results obtained in this section generalize the ones in \textup{\citep{OS111, BMS07, FMM10, JMN10, ElSi2011}}.
\end{rem}
In the case of a zero-sum game, we only have one value function for the players and therefore, Theorem \ref{theomainneccon1} becomes

\begin{thm}[Equivalent maximum principle for zero-sum game]\label{theomainneccon2}
Let $u\in \mathcal{A}$ with corresponding solutions $X(t)$ of \textup{\eqref{eqstateprocess1}}, $(Y(t),Z(t),K(t,\zeta),V(t))$ of \textup{\eqref{eqassBSDE}}, $A(t)$ of \textup{\eqref{lambda1}}, $(p(t),q(t),r(t,\zeta),w_i(t))$ of \textup{\eqref{ABSDE1}} and corresponding derivative processes $X_1(t)$ and $(y_1(t),z_1(t),k_1(t,\zeta),v_1(t))$ given by \textup{\eqref{derivstate1}} and \textup{\eqref{derivassBSDE1}}, respectively. Assume that conditions of \textup{Theorem \ref{theomainneccon1}} are satisfied. Then the following statements are equivalent:

\begin{enumerate}
\item \begin{align}\label{thnecondsum1}
\dfrac{\diffns}{\diffns \ell}J^{(u_1+\ell\beta_1,u_2)}(t)\Big. \Big|_{\ell=0}=\dfrac{\diffns}{\diffns \ell }J^{(u_1,u_2+\ell \beta_2)}(t)\Big. \Big|_{\ell=0}=0
\end{align}
for all bounded $\beta_1\in \mathcal{A}_1,\,\,\,\beta_2\in \mathcal{A}_2$.

\item \begin{align}\label{thnecondsum2}
E\Big[\dfrac{\partial H}{\partial \mu_1} (t,\mu_1(t),u_2(t))\Big. \Big| \mathcal{E}_t^{(1)}\Big]_{\mu_1=u_1(t)} =E\Big[\dfrac{\partial H}{\partial \mu_2} (t,u_1(t),\mu_2(t))\Big. \Big| \mathcal{E}_t^{(2)}\Big]_{\mu_2=u_2(t)}=0
\end{align}
for a.a $ t\in[0,T]$, where
\begin{align*}
&H(t,u_1(t),u_2(t))\\
&=H((t,X(t),\alpha(t),u_1,u_2,Y(t), Z(t), K(t,\cdot),V_1(t),A(t),p(t),q(t),r(t,\cdot),w(t)).
\end{align*}

\end{enumerate}

\end{thm}

\begin{proof}
It follows directly from Theorem \ref{theomainneccon1}.
\end{proof}

\begin{cor}
If $u=(u_1,u_2)\in \mathcal{A}_1\times \mathcal{A}_2$ is a Nash equilibrium for the zero-sum game in \textup{Theorem \ref{theomainneccon2}}, then equalities \textup{\eqref{thnecondsum2}} holds.

\end{cor}
\begin{proof}
If $u=(u_1,u_2)\in \mathcal{A}_1\times \mathcal{A}_2$ is a Nash equilibrium, then it follows from Theorem \ref{theomainneccon2} that \eqref{thnecondsum1} holds by \eqref{eqperfunctionzerosum111}.
\end{proof}

\section{Applications}\label{application}

\subsection{Application to robust utility maximization with entropy penalty}\label{Appli}

In this section, we apply the results obtained to study an utility maximization problem
under model uncertainty. We shall assume here that $\mathcal{E}^{(1)}_t=\mathcal{E}^{(2)}_t=\mathcal{F}_t$. The framework is that of \citep{BMS07}. We aim at finding a probability measure $Q\in \mathcal{Q}_{\mathcal{F}}$ that minimizes the functional
\begin{align}
E_{Q}\Big[\int_0^t a_0 S^{\kappa}(s)U_1(s)ds+\overline{a}_0 S^{\kappa} (T)U_2(T)\Big]+E_{Q}\Big[\mathcal{R}^{\kappa}(0,T)\Big ],
\end{align}
where $$\mathcal{Q}_F:=\Big\{Q|Q\ll
P,\ Q=P\, on\,\, \mathcal{F}_0\,\, and\,\, H(Q|P):=E_Q\Big[\ln \frac{\diffns Q}{\diffns P}\Big]\Big\} ,$$
with \\
$a_0$ and $\overline{a}_0$ being non-negative constants;\\
$\kappa=(\kappa(t))_{0\leq t\leq T}$ a non-negative bounded and progressively measurable;\\
$U_1=(U_1(t))_{0\leq t\leq T}$ a progressively measurable with $E_{P}\Big[\exp[\gamma_1\int_0^T|U_1(t)|\diffns t]\Big]<\infty, \, \forall \gamma_1>0$;\\
$U_2(T)$ a $\mathcal{F}_T-$measurable random variable with $E_{P}\Big[\exp[|\gamma_1 U_2(T)|]\Big]<\infty, \, \forall \gamma_1>0$;\\
$S^{\kappa}=\exp(-\int_0^t\kappa(s)\diffns s) $ is the discount factor and $\mathcal{R}^{\kappa}(t,T)$ is the penalization term, representing the sum
of the entropy rate and the terminal entropy, i.e
\begin{align}
\mathcal{R}^{\kappa}(t,T)=\frac{1}{S^{\kappa}(t)}\int_t^T\kappa(s)S^{\kappa}(s)\ln\frac{G_0^Q(s)}{G_0^Q(t)}\diffns s+\frac{S^{\kappa}(T)}{S^{\kappa}(t)}\ln\frac{G^Q(T)}{G_0^Q(t)}
\end{align}
with $G^Q=(G^Q(t))_{0\leq t\leq T}$ is the RCLL $P$-martingale representing the density of $Q$ with respect to $P$, i.e $$G^Q(t)=\frac{\diffns Q}{\diffns P}\Big|_{\mathcal{F}_t}$$
$G_T$ represents the Radon-Nikodym derivative on $\mathcal{F}_T$ of $Q$ with respect to $P$. More precisely

\begin{problem}\label{prob2} Find $Q^*\in \mathcal{Q}_{\mathcal{F}}$ such that
\begin{align}\label{probl_ap}
Y^{Q^*}(t)=\essinf_{Q\in \mathcal{Q}_{\mathcal{F}}} Y^Q(t)
\end{align}
with
 \begin{align}\label{exprY}
 Y^Q(t):= \frac{1}{S^{\kappa}(t)}E_{Q}\Big[\int_t^Ta_0 S^{\kappa }(s)U_1(s)\diffns s+\overline{a}_0S^{\kappa} (T)U_2(T)\Big|\mathcal{F}_t\Big]+E_{Q}\Big[\mathcal{R}^{\kappa}(t,T)\Big|\mathcal{F}_t\Big].
 \end{align}
\end{problem}
In the present regime switching jump-diffusion setup, we  consider the model uncertainty given by a probability measure $Q$ having a density $(G^{\theta}(t))_{0\leq t\leq T}$ with respect to $P$ satisfies the following SDE

\begin{equation}\left\{
\begin{array}{llll}
\diff G^{\theta}(t) & = & G^{\theta}(t)\Big[\theta_0(t)\diffns B(t)+\theta_1(t)\cdot \diffns \widetilde{\Phi}(t)+\displaystyle \int_{\mathbb{R}_0}\theta_2(t,\zeta)\,\widetilde{N}_\alpha(\diffns \zeta,\diffns t)\Big],\,\,\,\,\,\, t \in [ 0,T] \label{eqstateprocess_app} \\
G^{\theta}(0) &=& 1,
\end{array}\right.
\end{equation}
Using It\^o's formula, one can easily check that
\begin{eqnarray}
G^{\theta}(t)& = &\exp\Big[\int_0^t\theta_0(s)\diffns B(s)-\frac{1}{2}\int_0^t\theta^2_0(s)\diffns s+\int_0^t\int_{\mathbb{R}_0}\ln(1+\theta_2(\zeta,s))\widetilde{N}_\alpha(\diffns \zeta,\diffns s)\nonumber\\
&&+\int_0^t\int_{\mathbb{R}_0}\{\ln(1+\theta_2(s,\zeta))-\theta_2(s,\zeta)\}\nu_\alpha(\diffns \zeta)\diffns s+\sum_{j=1}^D \int_0^t\ln ( 1+\theta_{1,j}(s))\cdot \diffns \widetilde{\Phi}_j(s)\nonumber\\
&&+\sum_{j=1}^D \int_0^t\{\ln ( 1+\theta_{1,j}(s))-\theta_{1,j}(s)\}\lambda_j(s)\diffns s\Big].
\end{eqnarray}
Now, put $G^{\theta}(t,s)=\frac{G^{\theta}(s)}{G^{\theta}(t)},\, \, s\geq t$ then $(G^{\theta}(t,s))_{0\leq t\leq s\leq T}$ satisfies
\begin{equation}\left\{
\begin{array}{llll}
\diffns G^{\theta}(t,s) & = & G^{\theta}(t,s^-)\Big[\theta_0(s)\diffns B(s)+\theta_1(s)\cdot \diffns \widetilde{\Phi}(s)+\displaystyle \int_{\mathbb{R}_0}\theta_2(s,\zeta)\,\widetilde{N}_\alpha(\diffns s,\diffns \zeta)\Big],\,\,\, s \in [ t,T] \label{eqstateprocess_app2} \\
G^{\theta}(t,t) &=& 1.
\end{array}\right.
\end{equation}
Note that $\theta=(\theta_0,\theta_1,\theta_2)$ may be seen as a scenario control. Denote by $\mathcal{A}$ the set of all admissible controls $\theta=(\theta_0,\theta_1,\theta_2)$ such that

$$E\Big[\int_0^T\Big(\theta^2_0(t)+\sum_{j=1}^D\theta_{1,j}^2(t)\lambda_j(t)+\displaystyle \int_{\mathbb{R}_0}\theta_2^2(t,\zeta)\nu_\alpha(\diffns \zeta)\Big)\diffns t\Big]<\infty$$
 and $\theta_2(t,\zeta)\geq -1+\epsilon$ for some $\epsilon>0.$ \eqref{exprY} can be rewritten as
 \begin{align}\label{exprY2}
 Y^Q(t)=&E_{Q}\Big[\int_t^Ta_0e^{-\int_t^s\kappa(r)\diffns r}U_1(s)\diffns s+\overline{a}_0e^{-\int_t^T\kappa(r)\diffns r}U_2(T)\Big|\mathcal{F}_t\Big]\nonumber\\
 &+E_{Q}\Big[\int_t^T\kappa(s)e^{-\int_t^s\kappa(r)\diffns r}\ln G^{\theta}(t,s)ds+e^{-\int_t^T\kappa(r)\diffns r}\ln G^{\theta}(t,T)\Big|\mathcal{F}_t\Big]\nonumber\\
 =&E\Big[\int_t^T a_0 G^{\theta}(t,s)e^{-\int_t^s\kappa(r)\diffns r}U_1(s)\diffns s+\overline{a}_0 G^{\theta}(t,T)e^{-\int_t^T\kappa(r)\diffns r}U_2(T)\Big|\mathcal{F}_t\Big]\nonumber\\
  &+E\Big[\int_t^T\kappa(s)e^{-\int_t^s\kappa(r)\diffns r}G^{\theta}(t,s)\ln G^{\theta}(t,s)\diffns s\nonumber\\
&+e^{-\int_t^T\kappa(r)\diffns r}G^{\theta}(t,T)\ln G^{\theta}(t,T)\Big|\mathcal{F}_t\Big].
 \end{align}
Now, define $h_1$ by
\begin{align}\label{defh}
h_1(\theta(t)):=& \frac{1}{2} \theta_0^2(t)+\sum_{j=1}^D\{(1+\theta_{1,j}(t)\ln (1+\theta_{1,j}(t))-\theta_{1,j}\}\lambda_j(t)\nonumber\\
&+\int_{\mathbb{R}_0}\{(1+\theta_2(t,\zeta))\ln(1+\theta_2(t,\zeta))-\theta_2(t,\zeta)\}\nu_{\alpha}(\diffns \zeta).
\end{align}
Using  the It\^o-L\'evy  product rule, we have
\begin{align}\label{eqitolev11}
&E\Big[\int_t^T\kappa(s)e^{-\int_t^s\kappa(r)\diffns r}G^{\theta}(t,s)\ln G^{\theta}(t,s)\diffns s + e^{-\int_t^T\kappa(r)\diffns r}G^{\theta}(t,T)\ln G^{\theta}(t,T)\Big|\mathcal{F}_t\Big]\nonumber\\
=& E\Big[\int_t^Te^{-\int_t^s\kappa(r)\diffns r}G^{\theta}(t,s)h(\theta(s))ds\Big|\mathcal{F}_t\Big].
\end{align}
Substituting \eqref{eqitolev11} into \eqref{exprY2}, leads to
\begin{align}
Y^Q(t)=&E_t\Big[\int_t^T a_0 G^{\theta}(t,s)e^{-\int_t^s\kappa(r)\diffns  r}U_1(s)\diffns  s+\overline{a}_0 G^{\theta}(t,T)e^{-\int_t^T\kappa(r)\diffns  r}U_2(T)\Big]\nonumber\\
  &+E_t\Big[\int_t^T\kappa(s)e^{-\int_t^s\kappa(r)\diffns r}G^{\theta}(t,s)\ln G^{\theta}(t,s)\diffns s+e^{-\int_t^T\kappa(r)\diffns r}G^{\theta}(t,T)\ln G^{\theta}(t,T)\Big]\nonumber\\
  =&E_t\Big[\int_t^Te^{-\int_t^s\kappa(r)\diffns r}G^{\theta}(t,s)\Big(a_0 U_1(s)+h(\theta(s))\Big)\diffns s+\overline{a}_0 G^{\theta}(t,T)e^{-\int_t^T\kappa(r)\diffns r}U_2(T)\Big].
\end{align}

We have the following theorem
\begin{thm}
Suppose that the penalty function is given by \eqref{defh}. Then the optimal $Y^{Q^*}$ is such that $(Y^{Q^*},Z,W,K)$ is the unique solution to
the following quadratic BSDE

\begin{equation}\left\{
\begin{array}{llll}\label{eq_Y1}
\diff Y(t) & = & -\Big[ -\kappa(t)Y(t)+a U_1(t)-Z^2(t)+\sum_{j=1}^D\lambda_j(t)(-e^{W_j}-W_j+1)\\
&  & +\displaystyle \int_{\mathbb{R}_0}(-e^{-K(t,\zeta)}-K(t,\zeta)+1)\nu_{\alpha}\diffns \zeta\Big] \diffns t +Z(t) \diffns B(t)\\
& &+\sum_{j=1}^D W_j(t) \diffns \widetilde{\Phi}_j(t)+\displaystyle \int_{\mathbb{R}_0}K(t,\zeta)\widetilde{N}_{\alpha}(\diffns t,\diffns \zeta)\\
Y(T) &=& \overline{a}_0U_2(T),
\end{array}\right.
\end{equation}
Moreover, the optimal measure $Q^*$ solution of \textup{Problem \ref{prob2}} admits the Radon-Nikodym density $(G^{Q}(t,s))_{0\leq t\leq s\leq T}$ given by
\begin{equation}\left\{
\begin{array}{llll}\label{eq_Gtheta}
\diffns G^{\theta}(t,s) & = & G^{\theta}(t,s^-)\Big[-Z(s)\diffns B(s)+\sum_{j=1}^D(e^{-W_j}-1)\cdot \diffns \widetilde{\Phi})_j(s)\\
&&+\displaystyle \int_{\mathbb{R}_0}(e^{-K(s,\zeta)}-1)\,\widetilde{N}(\diffns s,\diffns \zeta)\Big],\,\,\, s \in [ t,T] \\
G(t,t) &=& 1.
\end{array}\right.
\end{equation}
\end{thm}

\begin{proof}
Fix $u_1$ and denote by $X(T)$ the corresponding wealth process.  One can see that \textup{Problem \ref{prob2}} can be obtained from our general control
problem by setting 
$X(t)=0,\, \forall t\in[0,T]$, $h(X(T),\alpha(T))=\overline{a}_0U_2(T)$, $f=0$, $\phi(x)=0$ and $\psi(x)=I$. Since  $h_1(\theta)$ given by (\ref{defh}) is convex in $\theta_0, \theta_1$ and $\theta_2$, it follows that conditions
 of Theorem \ref{mainressuf1} are satisfied. The Hamiltonian in this case is reduced to:
 \begin{equation}
 H(t,y,z,K,W)=\lambda (U_1(t)+h(\theta)+\theta_0 z)+\sum_{j=1}^D\lambda_j\theta_{1,j}W_j +\displaystyle \int_{\mathbb{R}_0}\theta_2(\cdot,\zeta)K(\cdot,\zeta)
\nu_{\alpha}(d \zeta)
 \end{equation}
Minimizing $H$ with respect to $\theta=(\theta_0,\theta_1,\theta_2)$ gives the first order condition of optimality for an optimal $\theta^\ast$,
\begin{equation}\left\{
\begin{array}{llll}\label{FOC}
\frac{\partial  H}{\partial \theta_0}=0 \,\,\,\text{ i.e., } \theta_0^{\ast}(t)=-Z(t),\\
\frac{\partial H}{\partial \theta_{1,j}}=0 \,\,\, \text{ i.e., }-\ln(1+\theta_{1,j}^\ast)(t)=-W_{1,j}(t)\text{ for } j=1,\ldots,D,\\
\nabla_{\theta_2}H=0  \,\,\,\text{ i.e., }-\ln(1+\theta_2^\ast)(t,\zeta)=-K(\cdot,\zeta),\,\,\,\,\nu_{\alpha}\text{- a.e.}
\end{array}\right.
\end{equation}
On the hand, one can show using product rule (see e.g., \citep{Menou20141}) that $Y$ given by \eqref{probl_ap2} is solution to the following linear BSDE
\begin{equation}\left\{
\begin{array}{llll}\label{eq_Y2}
\diff Y(t) & = & -\Big[ -\kappa(t)Y(t)+a U_1(t)+ h(\theta)+\theta_0Z(t)+\sum_{j=1}^D\theta_{1,j}(t)\lambda_j(t){W_j}\\
&  & +\displaystyle \int_{\mathbb{R}_0}\theta_2(t,\zeta)K(t,\zeta)\nu_{\alpha}\diffns \zeta\Big ]\diffns t+Z(t) dB(t)+ W(t)\cdot \diffns \widetilde{\Phi}(t)\\
& &+\displaystyle \int_{\mathbb{R}_0}K(t,\zeta)\widetilde{N}_{\alpha}(\diffns t,\diffns \zeta)\\
Y(T) &=& \overline{a}_0U_2(T),
\end{array}\right.
\end{equation}
Using comparison theorem for BSDE, $Q^\ast$ is an optimal measure for Problem \ref{prob2} if $\theta^\ast$ is such that
\begin{equation}\label{eq_g}
g(\theta^*)=\underset{\theta}\min \,g(\theta)
\end{equation}
for each $t$ and $\omega$, with $g(\theta):=h(\theta)+\theta_0Z(t)+\sum_{j=1}^D\theta_{1,j}(t)\lambda_j(t){W_j}+\displaystyle \int_{\mathbb{R}_0}\theta_2(t,\zeta)K(t,\zeta)\nu_{\alpha}\diffns \zeta$. This is equivalent to the first condition of optimality. Hence $(\theta^\ast_0,\theta^\ast_{1,1},\ldots,\theta^\ast_{1,D},\theta^\ast_2)$ satisfying \eqref{FOC} will satisfy \eqref{eq_g}. Substituting $\theta^\ast_0,\theta^\ast_{1,1},\ldots,\theta^\ast_{1,D},\theta^\ast_2$ into \eqref{eq_Y2} leads to \eqref{eq_Y1}. Furthermore, substituting $\theta^\ast_0,\theta^\ast_{1,1},\ldots,\theta^\ast_{1,D},\theta^\ast_2$ into \eqref{eqstateprocess_app2} gives \eqref{eq_Gtheta}. The proof of the theorem is complete.
\end{proof}

\begin{rem}
\leavevmode
\begin{itemize}
  \item This result can be seen as an extension to the Markov regime-switching setting of \textup{\citep[Theorem 1]{JMN10}} or \textup{\citep[Theorem 2]{BMS07}}.
  \item Let us mention that in the case $(X(t))_{0\leq t\leq T}$ is not zero and has a particular dynamics (mean-reverting or exponential Markov L\'evy switching) one can use \textup{Theorem \ref{mainressuf1}} to solve  a problem of recursive robust utility mazimization as in \textup{\citep[ Section 4.2]{OS111}} or \textup{\citep[Theorem 4.1]{Menou20141}}
\end{itemize}

\end{rem}

\subsection{Application to optimal investment of an insurance company under model uncertainty}

In this section, we use our general framework to study a problem of optimal investment of an insurance company under model uncertainty. The uncertainty here is also described by a family of probability measures. Such problem was solved in \citep{ElSi2011} using dynamic programming approach. We shall show that our general maximum principle enables us to solve the problem. We shall restrict ourselves to the case $\mathcal{E}_t^{(1)}=\mathcal{E}_t^{(2)}=\mathcal{F}_t,\,\,\,t\in [0,T]$.

The model is that of \citep[Section 2.1]{ElSi2011}. Let $(\Omega, \mathcal{F},P) $ be a complete probability space with $P$ representing a reference probability measure from which a family of real-world probability measures are generated. We shall suppose that $(\Omega, \mathcal{F},P) $ is big enough  to take into account uncertainties coming from future insurance claims, fluctuation of financial prices and structural changes in economics conditions. We consider a continuous-time Markov regime switching economic model with a bond and a stock or share index.

The evolution of the state of an economy over time is modeled by a continuous-time, finite-state, observable Markov chain $\alpha:=\{\alpha(t),t\in[0,T];\, T<\infty\}$ on $(\Omega, \mathcal{F},P)$, taking values in the state space $\mathbb{S}=\{e_1,e_2,\ldots,e_D\}$, where $D\geq 2$. We denote by $\Lambda:=\{\lambda_{nj}:1\leq n,j\leq D\}$ the intensity matrix of the Markov chain under $P$. Hence, for each $1\leq n,j\leq D,\,\,\lambda_{nj}$ is the transition intensity of the chain from state $e_n$ to state $e_j$ at time $t$. It is assumed that for $n\neq j,\,\,\lambda_{nj}> 0$ and $\sum_{j=1}^D \lambda_{nj}=0$, hence $\lambda_{nn}< 0$.
The dynamics of $(\alpha(t))_{0\leq t\leq T}$ is given in Section \ref{framew}.

 Let $r=\{r(t)\}_{t\in [0, T]}$ be the instantaneous interest rate of the money market account $B$ at time $t$. Then
\begin{equation}
r(t):=\langle \underline{r},\alpha(t)\rangle =\sum_{j=1}^D r_j\langle \alpha(t), e_j\rangle\;,
\end{equation}
where $\langle \cdot,\cdot\rangle$ is the usual scalar product in $\mathbb{R}^D$
and $\underline{r}=(r_1,\dots,r_D)\in \mathbb R^D_+$. Here the value $r_j$, the $j^{th}$ entry of the vector $\underline r$, represents the value of the interest rate when the Markov chain is in the state $e_j$, i.e., when $\alpha(t)=e_j$. The price dynamics of $B$ can now be written as
\begin{equation}\label{norisk-1}
\diffns S_0(t)=S_0r(t)\diffns  t,\, S_0(0)=1, \quad  t\in [0,T].
\end{equation}

Moreover, let $\mu=\{\mu(t)\}_{t\in [0, T]}$ and $\sigma=\{\sigma(t)\}_{t\in [0, T]}$ denote respectively the mean return and the volatility of the stock at time $t$. Using the same convention, we have
\begin{align*}
\mu(t)=&\langle \underline{\mu},\alpha(t)\rangle=\sum_{j=1}^D\mu_j\langle \alpha(t),e_j\rangle\;,\\
\sigma(t)=&\langle \underline{\sigma},\alpha(t)\rangle=\sum_{j=1}^D\sigma_j\langle \alpha(t),e_j \rangle\;,
\end{align*}
where
$$\underline{\mu}=(\mu_1,\mu_2,\ldots,\mu_D)\in \mathbb{R}^D,$$
and
$$\underline{\sigma}=(\sigma_1,\sigma_2,\ldots,\sigma_D)\in \mathbb{R_+}^D.$$

In a similar way, $\mu_j$ and $\sigma_i$ represent respectively the appreciation rate and volatility of the stock when the Markov chain is in state $e_j$, i.e., when $\alpha(t)=e_j$. Let $B=\{B_t\}_{t\in [0, T]}$ denotes the standard Brownian motion on $(\Omega,\mathcal{F},P)$ with respect  to its right-continuous complete filtration $\mathcal{F}^B:=\{\mathcal{F}^B_t\}_{0\leq t\leq T}$. Then, the dynamic of the stock price $S=\{S(t)\}_{t\in [0, T]}$ is given by the following Markov regime-switching geometric Brownian motion
\begin{equation}\label{risk-1}
\diffns S(t)=S(t)\left[\mu(t)\diffns t+\sigma(t)\diffns B(t)\right], \quad S(0)=S_0
\end{equation}
Let $Z_0:=\{Z_0(t)\}_{t\in [0, T]}$ be a real-valued Markov regime-switching pure jump process on $(\Omega,\mathcal{F},P)$. Here $Z_0(t)$ can be  considered as the aggregate amount of claims up to and including time $t$. Since $Z_0$ is a pure jump process, one has
\begin{equation}\label{pureJump}
Z_0(t)=\sum_{0<u\leq t}\Delta Z_0(u),\,\,Z_0(0)=0,\,\, P\text{-a.s, }  t\in[0,T]
\end{equation}
where for each $u\in[0,T]$, $\Delta_Z0(u)=Z_0(u)-Z_0(u^-)$, represents the jump size of $Z_0$ at time $u$.

 Assume that the state space  of claim size
 denoted by $\mathcal{Z}$ is $(0,\infty)$. Let $\mathcal{M}$ be the product space $[0,T]\times \mathcal{Z}$ of claim arrival time and claim size.
 Define a random measure $N^0(\cdot,\cdot)$ on the product space $\mathcal{M}$, which  selects claim arrivals and size $\zeta:=Z_0(u)-Z_0(u^-)$ at time $u$, then the aggregate insurance claim process $Z_0$ can be written as

\begin{equation}\label{pureprocess}
Z_0(t)=\int_0^t\int_0^{\infty} \zeta N^0(\diffns u,\diffns \zeta),\,\,\, t\in[0,T].
\end{equation}
Define, for each $t\in[0,T]$
\begin{equation}\label{pureprocessN}
N_{\Lambda^0}(t)=\int_0^t\int_0^{\infty}  N^0(\diffns u,\diffns \zeta),\,\,\,t\in[0,T].
\end{equation}
Then $N_{\Lambda^0}(t)$ counts the number of claim arrivals up to time $t$. Assume that, under the measure $P$, $N_{\Lambda^0}:=\{N_{\Lambda^0}(t)\}_{t\in [0, T]}$ is a conditional Poisson process on $(\Omega,\mathcal{F},P)$ with intensity $\Lambda^0:=\{\lambda^0(t)\}_{t\in [0, T]}$ modulated by the chain $\alpha$ given by
\begin{equation}
\lambda^0(t):=\langle \underline{\lambda}^0,\alpha(t)\rangle =\sum_{j=1}^D \lambda_j^0\langle \alpha(t), e_j\rangle\;,
\end{equation}
with $\underline{\lambda}^0=(\lambda_1^0,\ldots,\lambda^0_D)\in \mathbb R^D_+$. Here the value $\lambda_j^o$, the $j^{th}$ entry of the vector $\underline \lambda^0$, represents  the intensity rate of $N$ when the Markov chain is in the space state $e_j$, i.e., when $\alpha(t^-)=e_j$. Denote by $F_j(\zeta),\,j=1,\ldots,D$ the probability distribution of the claim size \\$\zeta:=Z_0(u)-Z_0(u^-)$ when  $\alpha(t^-)=e_j$. Then, the compensator of the Markov regime switching random measure $N^0(\cdot,\cdot)$ under $P$ is given by
\begin{align}
\nu_{\alpha}^0(\diffns u,\diffns \zeta):=\sum_{j=1}^D\langle \alpha(u^-),e_j\rangle\lambda_j^0F_j(\diffns\zeta)\diffns u.
\end{align}

Hence a compensated version $\widetilde{N}^0_{\alpha}(\cdot,\cdot)$ of the Markov regime-switching random measure is defined by
\begin{align}
\widetilde{N}^0_{\alpha}(\diffns u,\diffns \zeta)=N^0(\diffns u,\diffns \zeta)-\nu^0_{\alpha}(\diffns u,\diffns \zeta).
\end{align}
The premium rate $P_0(t)$ at time $t$ is given by
\begin{equation}
P_0(t):=\langle \underline{P_0},\alpha(t)\rangle =\sum_{j=1}^D P_{0,j}\langle \alpha(t), e_j\rangle,
\end{equation}
with $\underline{P_0}=(P_{0,1},\ldots,P_{0,D})\in \mathbb R^D_+$. Let $R_0:=\{R_0(t)\}_{t\in [0, T]}$ be the surplus process of the insurance company without investment. Then
\begin{align}
R_0(t):=&r_0+\int_0^tP_0(u)\diffns u -Z_0(t)\nonumber\\
=&r_0+\sum_{j=1}^D P_{0,j}\mathcal{J}_j(t)-\int_0^t\int_0^{\infty} \zeta N^0(\diffns  u,\diffns \zeta),\,\,\, t\in[0,T],
\end{align}
with $R_0(0)=r_0$. For each $j=1,\ldots,D$ and each $t\in[0,T]$, $\mathcal{J}_j(t)$ is the occupation time of the chain $\alpha$ in the state $e_j$ up to time $t$, that is
\begin{align}
\mathcal{J}_j(t)=\int_0^t\langle \alpha(u), e_j\rangle \diffns u.
\end{align}
The following information structure will be important for the derivation of the dynamic of the company' surplus process. Let
$\mathcal{F}^{Z_0}:=\{\mathcal{F}^{Z_0}\}_{0\leq t\leq T}$ denote the right-continuous $P$-completed  filtration generated by $Z_0$. For each $t\in [0,T]$
define $\mathcal{F}_t:=\mathcal{F}^{Z_0}_t\vee\mathcal{F}_t^{B}\vee\mathcal{F}_t^{\alpha}$ as the minimal $\sigma$-algebra generated by $\mathcal{F}^{Z_0}_t$,
$\mathcal{F}_t^{B}$ and $\mathcal{F}_t^{\alpha}$ and write $\mathbb{F}=\{\mathcal{F}_t\}_{0\leq t\leq T}$ as the information accessible to the company.

From now on, we assume that the insurance company invests the amount of $\pi(t)$ in the stock at time $t$, for each $t\in[0,T]$. Then $\pi=\{\pi(t),t\in[0,T]\}$ represents the portfolio process. Denote by $X=\{X^{\pi}(t)\}_{t\in [0, T]}$ the wealth process of the company. One can show that the dynamic of the surplus process is given by
\begin{equation}\left\{
\begin{array}{llll}\label{eq_X2}
\diff X(t) & = & \Big\{ P_0(t)+r(t)X(t)+ \pi(t)(\mu(t)-r(t))\Big\}\diffns t+\sigma(t)\pi(t)\diffns B(t)\\
&&-\displaystyle \int_0^{\infty} \zeta N^0(\diffns t,\diffns \zeta)\\
& = &  \Big\{ P_0(t)+r(t)X(t)+ \pi(t)(\mu(t)-r(t))-\displaystyle \int_0^{\infty}\zeta\nu^0_{\alpha}(\diff \zeta)\Big\}\diffns t\\
&&+\sigma(t)\pi(t)\diffns B(t)-\displaystyle \int_0^{\infty}\zeta \widetilde{N}_{\alpha}^0(\diffns t,\diffns \zeta),\,\,t\in [0,T,]\\
X(0) &=& X_0.
\end{array}\right.
\end{equation}
\begin{defi}
A portfolio $\pi$ is admissible if it satisfies
\begin{enumerate}
\item $\pi$ is $\mathbb{F}$-progressively measurable;
\item  \eqref{eq_X2} admits a unique strong solution;
\item $\sum_{j=1}^D E\Big[\int_0^T\Big\{|P_{0,j}+r_jX(t)+\pi(t)(\mu_j-r_j)|+\sigma^2_j\pi^2(t)+\lambda_j^0\int_0^{\infty}\zeta^2F_j(\diff \zeta)\Big\}\diffns t\Big]<\infty$;
\item $X(t)\geq 0,\,\,\forall t\in[0,T]$, $P$-a.s.
\end{enumerate}
We denote by $\mathcal{A}$ the space of all admissible portfolios.
\end{defi}

Define $\mathbb{G}:=\{\mathcal{G}_t, t\in[0,T]\}$, where $\mathcal{G}_t:=\mathcal{F}_t^B\vee \mathcal{F}_t^{Z_0}$, and for $n,j=1,\ldots,D$, let $\{C_{nj}(t),t\in[0,T]\}$ be a real-valued, $\mathbb{G}$-predictable, bounded, stochastic process on $(\Omega,\mathcal{F},P)$ such that for each $t\in[0,T]$
$C_{nj}\geq 0$ for $n\neq j$ and $\sum_{n=1}^D C_{nj}(t)=0,\, i.e,\,C_{nn}\leq 0$.

We consider a model uncertainty setup given by a probability measure $Q=Q^{\theta,\mathbf{C}}$ which is equivalent to $P$, with Radon-Nikodym derivative on $\mathcal{F}_t$ given by
\begin{align}
\frac{\diffns Q}{\diffns P}\Big|_{\mathcal{F}_t}=G^{\theta,C}(t),
\end{align}
where, for $0\leq t\leq T$, $G^{\theta,C}$ is a $\mathbb{F}$-martingale. Under $Q^{\theta,\mathbf{C}}$,  $\mathbf{C}:=\{\mathbf{C}(t),t\in[0,T]\}$ with $\mathbf{C}(t):=[C_{nj}(t)]_{n,j=1,\ldots,D}$ is a family of rate matrices of the Markov chain $\alpha(t)$; See e.g., \citep{DuEl1999}. For each $t\in [0,T]$, we set
 $$
  \mathbf{D}_0^{\mathbf{C}}(t):=\mathbf{D}^\mathbf{C}(t) -\mathbf{diag}(\mathbf{d}^C(t)),
 $$
  with
 $\mathbf{d}^C(t)=(d^C_{11},\ldots,d^C_{DD})^\prime\in \mathbb{R}^D$ and
\begin{align}
\mathbf{D}^C:=\Big[\frac{C_{nj}(t)}{\lambda_{nj}(t)}\Big]_{n,j=1,\cdots,D}=[d^{\mathbf{C}} _{nj}(t)].
\end{align}
We denote by $\mathcal{C}$ the space of all families intensity matrices $\mathbf{C}$ with bounded components.

 The Radon-Nikodym derivative or density process $G^{\theta,\mathbf{C}}$ is given by
\begin{equation}\left\{
\begin{array}{llll}\label{eq_Radon}
\diff G^{\theta,\mathbf{C}}(t) & = & G^{\theta,C}(t^-)\Big\{ \theta(t)\diffns B(t)+\displaystyle \int_0^{\infty}\theta(t)\widetilde{N}^0_{\alpha}(\diffns t,\diffns \zeta)\\
&&+(\mathbf{D}^{\mathbf{C}}_0(u)\alpha(u)-\mathbf{1})^\prime\cdot\diff \widetilde{\Phi}(t)\Big\},\,\,\,t\in[0,T],\\
 G^{\theta,\mathbf{C}}(0) &=& 1,
\end{array}\right.
\end{equation}
where $^\prime$ represents the transpose. Here $(\theta, \mathbf{C})$ may be regarded as scenario control. A control $\theta$ is admissible if $\theta$ is $\mathbb{F}$-progressively measurable, with $\theta(t)=\theta(t,\omega)\leq 1$ for a.a $(t,\omega)\in[0,T]\times\Omega$, and $\int_0^T\theta^2(t)\diffns t<\infty.$ We denote by $\Theta$ the space of such admissible processes.

Next, we formulate the optimal investment problem under model uncertainty. Let $U:(0,\infty)\longrightarrow \mathbb{R}$, be an utility function which is strictly increasing, strictly concave and twice continuously differentiable. The objectives of the insurance firm and the market are the following:
\begin{problem} \label{proinsur1}
Find a portfolio process $\pi^*\in \mathcal{A}$ and the process $(\theta^*, \mathbf{C}^\ast)\in \Theta\times \mathcal{C}$ such that
\begin{align}\label{eqpbAPP}
\underset{\pi \in  \mathcal{A}}{\sup}\,\,\underset{(\theta,\mathbf{C}) \in  \Theta\times\mathcal{C}}{\inf}\,E_{Q^{\theta,\mathbf{C}}}\Big[U^{\pi}(X_T)\Big]=&E_{Q^{\theta^\ast,\mathbf{C}^\ast}}
\Big[U^{\pi^\ast}(X_T)\Big]\notag\\
=&\underset{(\theta,\mathbf{C}) \in  \Theta\times\mathcal{C}}{\inf}\,\,\underset{\pi \in  \mathcal{A}}{\sup}\,E_{Q^{\theta,\mathbf{C}}}\Big[U^{\pi}(X_T)\Big].
\end{align}
\end{problem}
This problem can be seen as a zero-sum stochastic differential game problem of an insurance form. On one hand, we have
\begin{align}
E_{Q^{\theta,\mathbf{C}}}\Big[U^{\pi}(X_T)\Big]=E\Big[G^{\theta,\mathbf{C}}(T)U(X^{\pi}(T))\Big].
\end{align}
Now, define $Y(t)=Y^{\theta,\mathbf{C},\pi}(t)$ by
\begin{align}\label{eqYop11}
Y(t)=E\Big[\frac{G^{\theta,\mathbf{C}}(T)}{G^{\theta,\mathbf{C}}(t)}U(X^{\pi}(T))\Big|\mathcal{F}_t\Big].
\end{align}
Then, it can easily be shown that $Y(t)$ is the solution to the following linear BSDE
\begin{equation}\left\{
\begin{array}{llll}
\diffns Y (t) & = &-\Big[\theta(t)Z_0(t)+\displaystyle \int_{\mathbb{R}_0}\theta(t)K(t,\zeta)\nu^0_{\alpha}(\diff \zeta)+\sum_{j=1}^D(\mathbf{D}_0^{\mathbf{C}}(t)\alpha(t)-\mathbf{1})_j\lambda_jV_j(t)\Big]\diffns t \\
&  & Z_0(t)\diffns B(t) +\displaystyle \int_{\mathbb{R}_0}K(t,\zeta)\widetilde{N}^0_\alpha(\diffns \zeta,\diffns t)+V(t)\cdot  \diffns \widetilde{\Phi}(t),\,\,t\in [0,T], \label{eqstateprocessY} \\
Y(T) &=& U(X^{\pi}(T)).
\end{array}\right.
\end{equation}
Noting that
\begin{align}
Y(0)=Y^{\theta,\mathbf{C},\pi}(0)=E_{Q^{\theta,\mathbf{C}}}\Big[U^{\pi}(X_T)\Big],
\end{align}
Problem \ref{proinsur1} becomes
\begin{problem} \label{proinsur2}
Find a portfolio process $\pi^\ast\in \mathcal{A}$ and the process $(\theta^\ast, \mathbf{C}^\ast)\in \Theta\times \mathcal{C}$ such that
\begin{align}\label{eqpbAPP2}
\underset{\pi \in  \mathcal{A}}{\sup}\,\,\underset{(\theta,\mathbf{C}) \in  \Theta\times\mathcal{C}}{\inf}\,Y^{\theta,\mathbf{C},\pi}(0)=Y^{\theta^\ast,\mathbf{C}^\ast,\pi^\ast}(0)
=\underset{(\theta,\mathbf{C}) \in  \Theta\times\mathcal{C}}{\inf}\,\,\underset{\pi \in  \mathcal{A}}{\sup}\,Y^{\theta,\mathbf{C},\pi}(0),
\end{align}
where $Y^{\theta,\mathbf{C},\pi}$ is described by the forward-Backward system \eqref{eq_X2} and \eqref{eqstateprocessY}.
 \end{problem}

 \begin{thm}\label{thmappliopinv1}
Let $X^{\pi}(t)$ be dynamic of the surplus process satisfying \eqref{eq_X2}. Consider the optimization problem to find $\pi^\ast\in \mathcal{A}$ and $(\theta^\ast,\mathbf{C}^\ast)\in \Theta \times \mathcal{C}$ such that \eqref{eqpbAPP} \textup{(}or equivalently \eqref{eqpbAPP2}\textup{)} holds, with
\begin{align}\label{eqYop12}
Y^{\theta,\mathbf{C},\pi}(t)=E\Big[\frac{G^{\theta,\mathbf{C}}(T)}{G^{\theta,\mathbf{C}}(t)}U(X^{\pi}(T))\Big|\mathcal{F}_t\Big].
\end{align}
Moreover, suppose that $U(x)=-e^{-\beta x},\, \beta\geq 0.$ Then the optimal investment $\pi^\ast(t)$ and the optimal scenario measure of the market $(\theta^\ast,\mathbf{C}^\ast)$ are given respectively by
\begin{align}
\theta^\ast(t)=&-\sum_{n=1}^D\Big(\frac{\mu_n(t)-r_n(t)-\sigma^2_n(t)\pi^\ast(t,e_n)\beta}{\sigma_n}\Big)\langle \alpha(t),e_n\rangle,\\
 \pi^\ast(t)=&\sum_{n=1}^D\bigg(\frac{\displaystyle \int_{\mathbb{R}^+}(e^{\beta \zeta}-1)\lambda_n^0 F_n(\diffns \zeta)}{\beta\sigma_n}\bigg)\langle \alpha(t),e_n \rangle,
\end{align}
and the optimal $\mathbf{C}^\ast$ satisfies the following constraint linear optimization problem:

  \begin{align}
\underset{C_{1j},\ldots,C_{Dj}}\min \,\, \sum_{j=1}^D(\mathbf{D}^{\mathbf{C}}_0(t)e_n-\mathbf{1})_j\lambda_{nj} V_j(t)\,\,\, j=1,\ldots,D,
\end{align}
 subject to the linear constraints
 $$\sum_{n=1}^DC_{nj}(t)=0,$$
 where $V_j$ is given by \eqref{eq_V}.

Moreover, if we assume that the space of family matrix rates $(C_{nj})_{n,j=1,2}$ is bounded and write $C_{nj}(t)\in \Big[C^l(n,j), C^u(n,j)\Big]$ with $C^l(n,j)< C^u(n,j),\,\,n,j=1,2$. Then, in this case, the optimal $\mathbf{C}^\ast$ is given by:
\begin{eqnarray}
\mathbf{C}^\ast_{21}(t)&=&C^l(2,1)\mathbb{I}_{V_1(t)-V_2(t)>0}+C^u(2,1)\mathbb{I}_{V_1(t)-V_2(t)<0},\nonumber\\
\mathbf{C}^\ast_{11}(t)&=&-\mathbf{C}^\ast_{21}(t).
\end{eqnarray}
 The same we have that the solution for problem \ref{PG2} is given by:
 \begin{eqnarray}
\mathbf{C}^\ast_{21}(t)&=&C^l(2,1)\mathbb{I}_{V_2(t)-V_1(t)->0}+C^u(2,1)\mathbb{I}_{V_2(t)-V_1(t)<0},\nonumber\\
\mathbf{C}^\ast_{22}(t)&=&-\mathbf{C}^\ast_{12}(t).
\end{eqnarray}

\end{thm}

 \begin{proof}

One can see that this is a particular case of a zero-sum stochastic differential game of the forward-backward system  of the form \eqref{eqstateprocess1} and \eqref{eqassBSDE} with $\psi=Id$, $\varphi=f=0$ and $ h(x)=U(x).$ The Hamiltonian in Section \ref{maxiprinc1} is reduced to
\begin{align}\label{eqHamil}
&H(t, x,e_n y,z,k,v,\pi,\theta,a,p,q,r^0,w)\notag\\
=& a\Big[\theta z+\displaystyle \int_{\mathbb{R}^+}\theta k(t,\zeta)\nu^0_{e_n}(\diff \zeta)+\sum_{j=1}^D(\mathbf{D}_0^\mathbf{C}(t)e_n-\mathbf{1})_j\lambda_{nj}v_j(t)\Big] \nonumber\\
&+\Big[P_0(t)+rx+\pi(\mu-r)-\displaystyle \int_{\mathbb{R}^+}\zeta\nu^0_{e_n}(\diffns \zeta)\Big]p\nonumber\\
&+\sigma \pi q-\displaystyle \int_{\mathbb{R}^+}\zeta r^0(t,\zeta)\nu_{e_n}^0(\diffns \zeta).
\end{align}
The adjoint processes $A(t)$ ,$(p(t),q(t),r^0(t,\zeta),w(t))$ associated with the Hamiltonian are given by the following forward-backward SDE
\begin{equation}\left\{
\begin{array}{llll}
\diffns A(t)&=&A(t)\Big[\theta(t)\diffns B(t)+\displaystyle \int_{\mathbb{R}^+}\theta(t)\widetilde{N}^0_\alpha(\diffns \zeta,\diffns t)+(\mathbf{D}_0^{\mathbf{C}}(t)\alpha(t)-\mathbf{1})^\prime\cdot\diffns \widetilde{\Phi}(t)\Big],\,t\in[0,T], \\
A(0)&=&1, \label{eqA111}
\end{array}\right.
\end{equation}
and
\begin{equation}\left\{
\begin{array}{llll}
\diffns p(t)&=&-r(t)p(t)\diffns t+q(t)\diffns B(t)+\displaystyle \int_{\mathbb{R}^+}r^0(t,\zeta )\widetilde{N}^0_\alpha(\diffns \zeta,\diffns t)+W(t)\cdot\diffns \widetilde{\Phi}(t),\,t\in[0,T], \\
p(T)&=&A(T)U^\prime(X(T)).\label{eqq_P1}
\end{array}\right.
\end{equation}

It is clear that the functions $h, \phi$ and $H$ satisfy the assumptions  of Theorem \ref{theomainneccon2}.

Maximizing the Hamiltonian $H$ with respect to $\pi$ gives the first order condition for an optimal $\pi^\ast$.
\begin{align}\label{eqHami}
\frac{\partial H}{\partial \pi}=0\,\, \text{ i.e, }\,\, (\mu-r)p+\sigma q=0.
\end{align}
The BSDE \eqref{eqq_P1} is linear in $p$, hence we shall try a process $p(t)$ of the form
\begin{align}\label{eqp1111}
p(t)=\beta f(t,\alpha(t))A(t)e^{-\beta X(t)},
\end{align}
 where $f(\cdot,e_n)$ satisfies a differential equation to be determined. Applying the It\^o-L\'evy's formula for jump-diffusion, Markov regime-switching process (see e.g., \citep[Theorem 4.1]{ZES2012}), we get
 \begin{align}
 \diffns \Big( A(t)e^{-\beta X(t)}\Big)=&e^{-\beta X(t)}A(t)\Big[\theta(t)\diffns B(t)+\displaystyle \int_{\mathbb{R}^+}\theta(t)\widetilde{N}^0_\alpha(\diffns \zeta,\diffns t)+(\mathbf{D}_0^{\mathbf{C}}(t)\alpha(t)-\mathbf{1})^\prime\cdot\diffns \widetilde{\Phi}(t)\Big]\nonumber\\
  &+A(t) e^{-\beta X(t)}\Big[\Big( -\beta\Big\{ P_0(t)+r(t)X(t)+\pi(t)(\mu(t)-r(t))\Big\}+\frac{1}{2}\beta^2\sigma^2(t)\pi^2(t)\notag\\
  &
  +\displaystyle \int_{\mathbb{R}^+}(e^{\beta \zeta}-1) \nu^0_{\alpha}(\diffns \zeta)\Big)\diffns t-\beta \sigma(t)\pi(t)\diffns B(t)+\displaystyle \int_{\mathbb{R}^+}(e^{\beta \zeta}-1)\widetilde{N}^0_\alpha(\diffns \zeta,\diffns t)\Big]\nonumber\\
  &-\beta A(t)e^{-\beta X(t)}\theta(t)\sigma(t)\pi(t)\diffns t +\displaystyle \int_{\mathbb{R}^+}\theta(t)A(t) e^{-\beta X(t)}(e^{\beta \zeta}-1)N^0_\alpha(\diffns \zeta,\diffns t)
  \nonumber\\
  =& A(t)e^{-\beta X(t)}
  \Big[\Big( -\beta\Big\{ P_0(t)+r(t)X(t)+\pi(t)(\mu(t)-r(t))\Big\}-\beta \theta(t)\sigma(t)\pi(t)\notag\\
  &
  +\frac{1}{2}\beta^2\sigma^2(t)\pi^2(t)+\displaystyle \int_{\mathbb{R}^+}(1+\theta(t))(e^{\beta \zeta}-1) \nu^0_{\alpha}(\diffns \zeta)\Big)\diffns t\notag\\
&+ ( \theta(t)-\beta \sigma(t)\pi(t))\diffns B(t)+\displaystyle \int_{\mathbb{R}^+}\Big\{(1+\theta(t))(e^{\beta \zeta}-1)+\theta(t)\Big\}\widetilde{N}^0_\alpha(\diffns \zeta,\diffns t)\notag\\
 &+(\mathbf{D}_0^{\mathbf{C}}(t)\alpha(t)-\mathbf{1})^\prime\cdot\diffns \widetilde{\Phi}(t)\Big].
  \end{align}
 Putting $A(t)e^{-\beta X(t)}=P_1(t)$, then $p(t)=\beta f(t,\alpha(t))P_1(t)$ and using once more the It\^o-L\'evy's formula for jump-diffusion Markov regime-switching process, we get
 \begin{align}\label{expr_Ap}
 \diffns p(t)=& \beta \diffns \Big(f(t,\alpha(t))P_1(t)\Big)\nonumber\\
   =& \beta\Big [ f^\prime(t,\alpha(t))P_1(t)\diffns  t +f(t,\alpha(t))P_1(t)\Big[\Big( -\beta\Big\{ P_0(t)+r(t)X(t)+\pi(t)(\mu(t)-r(t))\Big\}\notag\\
   &
   -\beta \theta(t)\sigma(t)\pi(t)+\frac{1}{2}\beta^2\sigma^2(t)\pi^2(t)+\displaystyle \int_{\mathbb{R}^+}(1+\theta(t))(e^{\beta \zeta}-1) \nu^0_{\alpha}(\diffns \zeta)\Big)\diffns t\notag\\
   &+ ( \theta(t)-\beta \sigma(t)\pi(t))\diffns B(t)\Big]+ \sum_{j=1}^D\Big(f(t,e_j)-f(t,\alpha(t))\Big)
  P_1(t)(\mathbf{D}_{0,\alpha}^{\mathbf{C}}(t))^j\lambda_j(t)\diffns t\nonumber\\
    &+\displaystyle \int_{\mathbb{R}^+}f(t,\alpha(t))P_1(t)\Big\{(1+\theta(t))(e^{\beta \zeta}-1)+\theta(t)\Big\}\widetilde{N}^0_\alpha(\diffns \zeta,\diffns t)\nonumber\\
&+ \sum_{j=1}^DP_1(t)\Big(    f(t,e_j)(\mathbf{D}_{0,\alpha}^{\mathbf{C}}(t))^j -f(t,\alpha(t)) \Big) \diffns \widetilde{\Phi}_j(t)\Big],
 \end{align}
  where $(\mathbf{D}_{0,\alpha}^{\mathbf{C}}(t))^j=(\mathbf{D}_0^{\mathbf{C}}(t)\alpha(t))^j$. Comparing \eqref{expr_Ap}  with \eqref{eqq_P1}, by equating the terms in $\diffns t$, $\diffns B(t)$, $\widetilde{N}_\alpha(\diffns \zeta,\diffns t)$, and $\diffns \widetilde{\Phi}_j(t)$  $j=1,\ldots,D$, respectively, we get
 \begin{align}\label{eqq1111}
 	q(t)=(\theta^\ast-\beta\sigma(t)\pi^\ast(t))p(t).
 \end{align}
 Substituting this into \eqref{eqHami} gives,
\begin{align}\label{eqoptheta}
(\mu-r)p=&-\sigma\Big(\theta^\ast-\sigma(t)\pi^\ast(t)\beta\Big)p,\nonumber\\
\text{i.e., }\,\, \theta^\ast(t)=&-\sum_{j=1}^D\Big(\frac{\mu_j(t)-r_j(t)-\sigma^2_j(t)\pi^\ast(t,e_j)\beta}{\sigma_j}\Big)\langle \alpha(t),s_j\rangle.
\end{align}
On the other hand, we also have
\begin{align}\label{eqr1111}
r^0(t,\zeta)=p(t)\Big\{(1+\theta(t))(e^{\beta \zeta}-1)+\theta(t)\Big\}
\end{align}
and
\begin{align}\label{eqw1111}
w_j(t)=\beta\Big\{P_1(t)\Big(    f(t,e_j)(\mathbf{D}_{0,\alpha}^{\mathbf{C}}(t))^j -f(t,\alpha(t)) \Big)\Big\},
\end{align}
with the function $f(\cdot,e_n)$ satisfying the following backward differential equation:

\begin{align}
& f^\prime(t,e_n) +f(t,e_n)\Big[-\beta\Big\{ P_0(t,e_n)+r(t,e_n)X(t)+\pi(t)(\mu(t,e_n)-r(t,e_n))\Big\} + r(t,e_n)\notag\\
 &
 -\beta \theta(t)\sigma(t,e_n)\pi(t)+\frac{1}{2}\beta^2\sigma^2(t,e_n)\pi^2(t)+\displaystyle \int_{\mathbb{R}^+}(1+\theta(t))(e^{\beta \zeta}-1) \lambda^0_{n}F_{e_n}(\diffns \zeta)\Big] \nonumber\\
  &+\sum_{j=1}^D\Big(f(t,e_j)-f(t,e_n)\Big)(\mathbf{D}_{0,e_n}^{\mathbf{C}}(t))_{nj}\lambda_{n j}=0, \label{eqfopt11}
\end{align}
with the terminal condition $f(T,e_n)=1$, for $n=1,\ldots, D$. The solution of such backward equation can be found in \citep{ElSi2011}. Minimizing the Hamiltonian $H$ with respect to $\theta$ gives the first order condition for an optimal $\theta ^\ast$.

\begin{align}\label{eqHami2}
\frac{\partial H}{\partial \theta}=0\,\, \text{ i.e., }\,\, z+\displaystyle \int_{\mathbb{R}^+}k(t,\zeta)\nu^0_{\alpha}(\diffns \zeta)=0.
\end{align}

The BSDE \eqref{eqstateprocessY} is linear in $Y$, hence we shall try the process $Y(t)$ of the form
\begin{align}\label{eqY1111}
Y(t)=f_1(t,\alpha(t))Y_1(t)\,\, \text{ with }\,\, Y_1(t)=e^{-\beta X(t)},
\end{align}
where $f_1(\cdot,e_n),\,\,n=1,\ldots, D$ is a deterministic function satisfying a backward differential equation to be determined. Applying the It\^o-L\'evy's formula for jump-diffusion Markov regime-switching, we get

\begin{align}\label{eq_Y3}
\diffns Y(t)=& f_1^\prime(t,\alpha(t) )e^{-\beta X(t)}\diffns t- f_1(t, \alpha(t))Y_1(t)\beta\Big\{p_0(t)+r(t)X(t)\nonumber\\
 &+\pi(t)(\mu(t)-r(t))-\frac{1}{2}\beta\sigma^2(t)\pi^2(t)+\frac{1}{\beta}\displaystyle \int_{\mathbb{R}^+}(e^{\beta \zeta}-1)\nu^0_{\alpha}(\diffns \zeta)\Big\}\diffns t\nonumber\\
  &+\sum_{j=1}^D\Big(f_1(t,e_j)-f_1(t,\alpha(t))\Big)Y_1(t)\lambda_j(t)\diffns t\nonumber\\
 &- f_1(t,\alpha(t))Y_1(t)\beta\sigma(t)\pi(t)\diffns B(t)+ \displaystyle \int_{\mathbb{R}^+}f_1(t,\alpha(t))Y_1(t)(e^{\beta \zeta}-1)\widetilde{N}^0_\alpha(\diffns \zeta,\diffns t)\nonumber\\
 &+ \sum_{j=1}^D\Big(f_1(t,e_j)-f_1(t,\alpha(t))\Big)Y_1(t) \diffns \widetilde{\Phi}_j(t).
\end{align}
Comparing \eqref{eq_Y3} and \eqref{eqstateprocessY}, we get
\begin{align}
Z(t)=&-\beta Y(t)\sigma(t)\pi(t),\label{eq_z}\\
K(t,\zeta)=& Y(t)(e^{\beta \zeta}-1),\label{eq_K}\\
V_j(t)=&\Big\{f_1(t,e_j)-f_1(t,\alpha(t)\Big\}Y_1(t)\label{eq_V}.
\end{align}
Substituting $Z(t)$ and $K(t,\zeta)$ into \eqref{eqHami2}, we get
\begin{align}\label{eqoppi1}
\beta\sigma(t)\pi^\ast(t)=&\displaystyle \int_{\mathbb{R}^+}(e^{\beta \zeta}-1)\nu^0_{\alpha}(\diffns \zeta),\nonumber\\
\text{ i.e., }\,\, \pi^\ast(t)=&\sum_{n=1}^D\bigg(\frac{\displaystyle \int_{\mathbb{R}^+}(e^{\beta \zeta}-1)\lambda_n^0 F_n(\diffns \zeta)}{\beta\sigma_n}\bigg)\langle \alpha(t),e_n \rangle.
\end{align}
Substituting \eqref{eq_z}-\eqref{eq_V} into \eqref{eq_Y3}, we deduce that the function $f_1(\cdot, e_n)$ satisfies the following backward differential  equation
\begin{align}\label{eqq_f1}
& f_1^\prime(t,e_n) +f_1(t,e_n)\Big[-\beta\Big\{ P_0(t,e_n)+r(t,e_n)X(t)+\pi(t)(\mu(t,e_n)-r(t,e_n))\Big\} + r(t,e_n)\notag\\
&
-\beta \theta(t)\sigma(t,e_n)\pi(t)+\frac{1}{2}\beta^2\sigma^2(t,e_n)\pi^2(t)+\displaystyle \int_{\mathbb{R}^+}(1+\theta(t))(e^{\beta \zeta}-1) \lambda^0_{n}F_{e_n}(\diffns \zeta)\Big] \nonumber\\
&+\sum_{j=1}^D\Big(f_1(t,e_j)-f_1(t,e_n)\Big)(\mathbf{D}_{0,e_n}^{\mathbf{C}}(t))_{nj}\lambda_{n j}=0
\end{align}
with the terminal condition $f_1(T,e_n)=-1$ for $n=1,\ldots,D$. 

As for the optimal $(C_{nj})_{n,j=1,\ldots,D}$, the only part of the Hamiltonian that depends on $\mathbf{C}$ is the sum $\sum_{j=1}^D(\mathbf{D}^{\mathbf{C}}_0(t)e_n-\mathbf{1})_j\lambda_{nj} V_j(t)$, Hence minimizing  the Hamiltonian with respect to $\mathbf{C}$ is equivalent to minimizing the following system of differential operator
\begin{align}\label{eqc111}
\underset{C_{1j},\ldots,C_{Dj}}\min \,\, \sum_{j=1}^D(\mathbf{D}^{\mathbf{C}}_0(t)e_n-\mathbf{1})_j\lambda_{nj} V_j(t)\,\,\, j=1,\ldots,D,
\end{align}
 subject to the linear constraints
 $$\sum_{n=1}^DC_{nj}(t)=0.$$
Hence, one can obtain the solution in the two-states case(since $C$ is bounded ) with $V_j$ and $f_1$ given by \eqref{eq_V} and \eqref{eqq_f1} respectively. More specifically, when the Markov has only two states, we have to solve the following two linear programming problems:

\begin{align}\label{PG1}
\underset{C_{11}(t),C_{21}(t)}\min \,\, ( V_1(t)-V_2(t))C_{21}(t)+\lambda_{21} (V_2(t)-V_1(t))
\end{align}
subject to the linear constraint $$C_{11}+C_{21}=0.$$

and
\begin{align}\label{PG2}
\underset{C_{12}(t),C_{22}(t)}\min \,\, ( V_2(t)-V_1(t))C_{12}(t)+\lambda_{12} (V_1(t)-V_2(t))
\end{align}
subject to the linear constraint $$C_{12}+C_{22}=0.$$

By imposing that the space of family matrix rates $(C_{nj})_{n,j=1,2}$ is bounded we can write that $C_{nj}(t)\in \Big[C^l(n,j), C^u(n,j)\Big]$ with $C^l(n,j)< C^u(n,j),\,\,i,j=1,2$. The solution to the preceding two linear control problems are then given by:
\begin{eqnarray}
C^\ast_{21}(t)&=&C^l(2,1)\mathbb{I}_{V_1(t)-V_2(t)>0}+C^u(2,1)\mathbb{I}_{V_1(t)-V_2(t)<0},\nonumber\\
C^\ast_{11}(t)&=&-C^\ast_{21}(t).
\end{eqnarray}
 The same we have that the solution for problem \ref{PG2} is given by:
 \begin{eqnarray}
C^\ast_{21}(t)&=&C^l(2,1)\mathbb{I}_{V_2(t)-V_1(t)->0}+C^u(2,1)\mathbb{I}_{V_2(t)-V_1(t)<0},\nonumber\\
C^\ast_{22}(t)&=&-C^\ast_{12}(t).
\end{eqnarray}
The proof is completed

\end{proof}

\begin{rem} \leavevmode
	\begin{itemize}
		\item Assume for example that the distribution of the claim size is of exponential type \textup{(}with parameter  $\tilde{\lambda}_j^0>2\beta,\,j=1,\ldots,n$\textup{)}. Moreover, assume that  $\pi, \theta$ and $C$ are given by \eqref{eqoppi1}, \eqref{eqoptheta} and \eqref{eqc111}, respectively. Then each of the following equations: Eq. \eqref{eq_X2}, \eqref{eqstateprocessY}, \eqref{eqA111} and \eqref{eqq_P1} admits a unique solution. The solution $(\widehat{Y}(t), \widehat{Z}(t), \widehat{K}(t,\zeta), \widehat{V}(t))$ \textup{(}respectively $(\widehat{p}_i(t),\widehat{q}_i(t),\widehat{r}_i(t,\zeta),\widehat{w}_i(t))$\textup{)} to \eqref{eqstateprocessY}  \textup{(}respectively \eqref{eqq_P1}\textup{)} is given by \eqref{eqY1111}, \eqref{eq_z}, \eqref{eq_K} and \eqref{eq_V} \textup{(}respectively \eqref{eqp1111}, \eqref{eqq1111}, \eqref{eqr1111} and \eqref{eqw1111}\textup{)}.
		
		\item We note that $f$ given by \eqref{eqfopt11} and $f_1$ given by \eqref{eqq_f1} coincide. Moreover, for $r=0$, the backward differential equation \eqref{eqfopt11}is the same as \textup{\citep[Eq. (4.13)]{ElSi2011}}
	\end{itemize}

\end{rem}

\section{Conclusion}\label{conlrem}

In this paper, we use a general maximum principle for Markov regime-switching forward-backward stochastic differential equation to study optimal strategies for stochastic differential games. The proposed model covers the model uncertainty in \cite{BMS07, ElSi2011, FMM10, JMN10, OS111}. The results obtained are applied to study two problems: First, we study robust utility maximization under relative entropy penalization. We show that the value function in this case is described by a quadratic backward stochastic differential equation. Second, we study a problem of optimal investment of an insurance company under model uncertainty. This can be formulated as a two-player, zero-sum, stochastic differential games between the market and the insurance company, where the market controls the mean relative growth rate of the risky asset and the company controls the investment. We find ``closed form" solutions of the optimal strategies of the insurance company and the market, when the utility is of exponential type and the Markov chain has two states.

Optimal control for delayed systems has received attention recently, due to the memory dependence of some processes. In this situation, the dynamic at the present time $t$ does not only depend on the situation at time $t$ but also on a finite part of their past history. Extension of the present work to the delayed case could be of interest. Such results were derived in \cite{Menou20141}, in the case of no regime-switching.

\appendix
\numberwithin{equation}{section}
\section{}

\begin{proof}[Proof of Theorem \protect\ref{mainressuf1}]
Let show that
$J_1(u_1,\widehat{u}_2,e_n)\leq J_1(\widehat{u}_1,\widehat{u}_2,e_n) \text{ for all } \widehat{u}_1\in \mathcal{A}_1.$
Fix $u_1  \in  \mathcal{A}_1$, then,  we have
\begin{align}\label{eqdiff1}
J_1(u_1,\widehat{u}_2,e_n)- J_1(\widehat{u}_1,\widehat{u}_2,e_n)=I_1+I_2+I_3,
\end{align}
where
\begin{align}
I_1=&E\Big[ \int_0^T \Big\{ f_1(t,X(t),\alpha(t),Y(t),u(t))-f_1(t, \widehat{X}(t),\alpha(t),\widehat{u}(t)) \Big\}\diff t \Big], \label{eqI11}\\
I_2=&E\Big[ \varphi_1(X(T),\alpha(T)) -\varphi_1(\widehat{X}(T), \alpha(T))\Big],\label{eqI21}\\
I_3=&E\Big[\psi_1(Y_1(0))\,-\,\psi_1(\widehat{Y_1}(0))\Big]. \label{eqI31}
\end{align}
By the definition of $H_1$, we get
\begin{align}
I_1=&E\Big[ \int_0^T \Big\{ H_1(t,u(t))-\widehat{H}_1(t,\widehat{u}(t))- \widehat{A}_1(t)(g_1(t)-\widehat{g}_1(t))-\widehat{p}_1(t)(b(t)-\widehat{b}(t))\Big.\Big.\notag\\
&\Big.\Big. -\widehat{q}_1(t)(\sigma(t)-\widehat{\sigma}(t))-\int_{\mathbb{R}_0}\widehat{r}_1(t,\zeta)(\gamma(t,\zeta)-\widehat{\gamma}(t,\zeta))\nu_\alpha(\diff\zeta) \notag\\
  &-\sum_{j=1}^D\widehat{w_1}^j(t)(\eta_j(t)-\widehat{\eta}_j(t) )\lambda_{j}(t)   \Big\} \diff t\Big]. \label{eqI12}
\end{align}
By concavity of $\varphi_1$ in $x$, It\^o formula,  and \eqref{ABSDE1} we have
\begin{align}\label{eqI22}
I_2\leq& E\Big[\dfrac{\partial \varphi_1}{\partial x} (\widehat{X}(T),\alpha(T))(X(T)-\widehat{X}(T)) \Big]\notag\\
=& E\Big[\widehat{p}_1(T) (X(T)-\widehat{X}(T))\Big]-E\Big[\widehat{A}_1(T)\dfrac{\partial h_1}{\partial x}(\widehat{X}(T),\alpha(T))(X(T)-\widehat{X}(T))\Big]\notag\\
=& E\Big[\int_0^T\widehat{p}_1(t) (\diffns X(t)-\diffns \widehat{X}(t))+\int_0^T(X(t^-)-\widehat{X}(t^-))
\diff\widehat{p}_1(t)+\int_0^T(\sigma(t)-\widehat{\sigma}(t))\widehat{q}_1(t)\diff t \Big.\notag\\
&+\int_0^T\int_{\mathbb{R}_0}(\gamma(t)-\widehat{\gamma}(t))\widehat{r}_1(t,\zeta)\nu_\alpha(\diffns\zeta)\diff t+  \int_0^T \sum_{j=1}^D\widehat{w_1}^j(t)(\eta_j(t)-\widehat{\eta}_j(t) )\lambda_{j}(t)    \diff t    \Big]\notag\\
&-E\Big[\widehat{A}_1(T)\dfrac{\partial h_1}{\partial x} (\widehat{X}(T),\alpha(T))(X(T)-\widehat{X}(T))\Big]\notag\\
=& E\Big[\int_0^T\widehat{p}_1(t) (b(t)-\widehat{b}(t))\diff t+\int_0^T(X(t^-)-\widehat{X}(t^-))
  \Big(-\frac{\partial \widehat{H}_1}{\partial x}(t)\Big)\diff t +\int_0^T(\sigma(t)-\widehat{\sigma}(t))\widehat{q}(t)\diff t \Big.\notag\\
&+\int_0^T\int_{\mathbb{R}_0}(\gamma(t)-\widehat{\gamma}(t))\widehat{r}_1(t,\zeta)\nu_\alpha(\diffns\zeta)\diff t+  \int_0^T \sum_{j=1}^D\widehat{w_1}^j(t)(\eta_j(t)-\widehat{\eta}_j(t) )\lambda_{j}(t)    \diff t    \Big]\notag\\
&-E\Big[\widehat{A}_1(T)\dfrac{\partial h_1}{\partial x} (\widehat{X}(T),\alpha(T))(X(T)-\widehat{X}(T))\Big].
\end{align}

By concavity of $\psi_1,h_1$, It\^o formula, \eqref{eqassBSDE} and \eqref{lambda1}, we get
\begin{align}
I_3\leq & E\Big[\psi_1^\prime(\widehat{Y}_1(0))(Y_1(0)-\widehat{Y}_1(0))\Big]\notag\\
=& E\Big[\widehat{A}_1(0)(Y_1(0)-\widehat{Y}_1(0))\Big]\notag\\
=& E\Big[\widehat{A}_1(T)(Y(T)-\widehat{Y}_1(T))\Big]-E\Big[\int_0^T\widehat{A}_1(t) (\diff Y_1(t)-\diff  \widehat{Y}_1(t))\Big.\notag\\
& +\int_0^T(Y_1(t^-)-\widehat{Y}_1(t^-))\diff \widehat{A}_1(t) + \int_0^T(Z_1(t)-\widehat{Z}_1(t))\dfrac{\partial \widehat{H}_1}{\partial z}(t) \diff t \notag\\
&\Big. +\int_0^T\int_{\mathbb{R}_0}(K_1(t,\zeta)-\widehat{K}_1(t,\zeta))\nabla_k \widehat{H}_1(t,\zeta)\nu_\alpha(\diff \zeta)\diff t+\int_0^T \sum_{j=1}^D\dfrac{\partial \widehat{H}_1}{\partial v^j}(t)  (V_1^j(t)-\widehat{V}_1^j(t) )\lambda_{j}(t)    \diff t    \Big]\notag\\
=& E\Big[\widehat{A}_1(T) \{h_1(X(T),\alpha(T))-h_1(\widehat{X}(T),\alpha(T))\} \Big]-E\Big[\int_0^T\dfrac{\partial \widehat{H}_1}{\partial y}(t) (Y_1(t)-\widehat{Y}_1(t)) \diff t \Big.\notag\\
&+\int_0^T\widehat{A}_1(t) (-g(t)+\widehat{g}(t))\diff t + \int_0^T(Z(t)-\widehat{Z}(t))\dfrac{\partial \widehat{H}}{\partial z}(t) \diff t\notag\\
& \Big.+\int_0^T\int_{\mathbb{R}_0}(K_1(t,\zeta)-\widehat{K}_1(t,\zeta))\nabla_k \widehat{H}_1(t,\zeta)\nu_\alpha(d\zeta)\diff t+\int_0^T \sum_{j=1}^D\dfrac{\partial \widehat{H}_1}{\partial v^j}(t)  (V_1^j(t)-\widehat{V}_1^j(t) )\lambda_{j}(t)    \diff t  \Big]. \notag
\end{align}
Hence we get
\begin{align}\label{eqI32}
I_3\leq& E\Big[\widehat{A}_1(T)\dfrac{\partial h_1}{\partial x}( \widehat{X}(T),\alpha(T))(X(T)-\widehat{X}(T)) \Big]-E\Big[\int_0^T\dfrac{\partial \widehat{H}_1}{\partial y}(t) (Y_1(t)-\widehat{Y}_1(t)) \diff t \Big.\notag\\
&+\int_0^T\widehat{A}_1(t) (-g(t)+\widehat{g}_1(t))\diff t + \int_0^T(Z_1(t)-\widehat{Z}_1(t))\dfrac{\partial \widehat{H}_1}{\partial z}(t) \diff t\\
& \Big.+\int_0^T\int_{\mathbb{R}_0}(K_1(t,\zeta)-\widehat{K}_1(t,\zeta))\nabla_k \widehat{H}_1(t,\zeta)\nu_\alpha(d\zeta)\diff t+\int_0^T \sum_{j=1}^D\dfrac{\partial \widehat{H}_1}{\partial v^j}(t)  (V_1^j(t)-\widehat{V}_1^j(t) )\lambda_{j}(t)    \diff t  \Big].\notag
\end{align}
Summing \eqref{eqI12}-\eqref{eqI32} up, we have
\begin{align}
I_1+I_2+I_3\leq & E\Big[ \int_0^T \Big\{H_1(t,u(t))-\widehat{H}_1(t,\widehat{u}(t))-\dfrac{\partial \widehat{H}_1}{\partial x}(t)(X(t)-\widehat{X}(t)) -\dfrac{\partial \widehat{H}_1}{\partial y}(t)(Y_1(t)-\widehat{Y}_1(t)) \Big. \Big. \notag\\
&\Big. \Big. +\int_{\mathbb{R}_0}(K_1(t,\zeta)-\widehat{K}_1(t,\zeta))\nabla_k \widehat{H}_1(t,\zeta)\nu_\alpha(d\zeta)\diff t\notag\\
&+ \sum_{j=1}^D\dfrac{\partial \widehat{H}_1}{\partial v^j}(t)  (V_1^j(t)-\widehat{V}_1^j(t) )\lambda_{j}(t) \Big\} \diffns t\Big].\label{eqdiff2}
\end{align}
One can show, using the same arguments in \cite{FOS05} that, the right hand side of \eqref{eqdiff2} is non-positive.  For sake of completeness we shall give the details here. Fix $t\in[0,T]$. Since $\widetilde{H}_1(x,y,z,k,v)$ is concave, it follows by the standard hyperplane argument (see e.g \cite[Chapter 5, Section 23]{Rock70}) that there exists a subgradient $d=(d_1,d_2,d_3,d_4(\cdot),d_5) \in \mathbb{R}^3\times \mathcal{R}\times \mathbb{R}$ for $\widetilde{H}_1(x,y,z,k,v)$ at $x=\widehat{X}(t),\,y=\widehat{Y}_1(t),\,z=\widehat{Z}_1(t),\,k=\widehat{K}_1(t,\cdot),\,v=\widehat{V}_1(t)$ such that if we define

\begin{align}\label{eqdiff3}
i_1(x,y,z,k,v):=&\widetilde{H}_1(x,y,z,k,v)-\widehat{H}_1(t)-d_1(x-\widehat{X}(t))-d_2(y-\widehat{Y}_1(t))-d_3(z-\widehat{Z}_1(t))\notag\\
&-\int_{\mathbb{R}_0} d_4(\zeta)(k(\zeta)-\widehat{K}_1(t,\zeta))\nu_\alpha(\diffns \zeta)-\sum_{j=1}^Dd_5^j (V_1^j(t)-\widehat{V}_1^j(t) )\lambda_{j}(t) .
\end{align}
Then $i(x,y,z,k,v)\leq 0$ for all $x,y,z,k,v$.

Furthermore, we clearly have $i(\widehat{X}(t),\widehat{Y}_1(t),\widehat{Z}_1(t),\widehat{K}_1(t,\cdot),\widehat{V}_1(t))$. It follows that,
\begin{align*}
d_1=\frac{\partial \widetilde{H}_1}{\partial x}(\widehat{X}(t),\widehat{Y}_1(t),\widehat{Z}_1(t),\widehat{K}_1(t,\cdot),\widehat{V}_1(t)),\\
d_2=\frac{\partial \widetilde{H}_1}{\partial y}(\widehat{X}(t),\widehat{Y}_1(t),\widehat{Z}_1(t),\widehat{K}_1(t,\cdot),\widehat{V}_1(t)),\\
d_3=\frac{\partial \widetilde{H}_1}{\partial z}(\widehat{X}(t),\widehat{Y}_1(t),\widehat{Z}_1(t),\widehat{K}_1(t,\cdot),\widehat{V}_1(t)),\\
d_4=\nabla_k \widetilde{H}_1(\widehat{X}(t),\widehat{Y}_1(t),\widehat{Z}_1(t),\widehat{K}_1(t,\cdot),\widehat{V}_1(t)),\\
d_5^j=\frac{\partial \widetilde{H}_1}{\partial v^j}(\widehat{X}(t),\widehat{Y}_1(t),\widehat{Z}_1(t),\widehat{K}_1(t,\cdot),\widehat{V}_1(t)).
\end{align*}
Combining this with \eqref{eqdiff2},  and using  the concavity of $\widetilde{H}_1$, we conclude that\\
$J_1(u_1,\widehat{u}_2,e_i)\leq J_1(\widehat{u}_1,\widehat{u}_2,e_i) \text{ for all } u_1 \in \mathcal{A}_1.$
In a similar way, one can show that $J_2(\widehat{u}_1,u_2,e_i)\leq J_2(\widehat{u}_1,\widehat{u}_2,e_i) \text{ for all } u_2 \in \mathcal{A}_2$. This completed the proof.
\end{proof}

\begin{proof}[Proof of Theorem \protect\ref{theomainneccon1}]
We have that
\begin{align}
&\dfrac{\diffns }{\diffns \ell}J_1^{(u_1+\ell \beta_1,u_2)}(t)\Big. \Big|_{\ell=0}\notag\\
=&E\Big[\int_0^T \Big\{\dfrac{\partial f_1}{\partial x}(t)X_1(t)  +\dfrac{\partial f_1}{\partial u_1}(t)\beta_1(t)\Big\} \diffns t+\dfrac{\partial \varphi_1}{\partial x}  (X^{(u_1,u_2)}(T),\alpha(T))X_1(T)+\psi_1^\prime(Y_1(0))y_1(0)\Big]\notag\\
=&J_1+J_2+J_3, \label{eqIprim11}
\end{align}
with
\begin{align*}
J_1=&E\Big[\int_0^T \Big\{\dfrac{\partial f_1}{\partial x}(t)X_1(t)  +\dfrac{\partial f_1}{\partial u_1}(t)\beta_1(t)\Big\} \diffns t\Big], \notag\\
J_2=&E\Big[\dfrac{\partial \varphi_1}{\partial x}  (X^{(u_1,u_2)}(T),\alpha(T))X_1(T)\Big], \notag\\
J_3=&E\Big[\psi_1^\prime(Y_10))y_1(0)\Big] .
\end{align*}

By It\^o's formula, \eqref{ABSDE1},  \eqref{derivstate1} and \eqref{thneccond1}, we have

\begin{align}
J_2=&E\Big[\dfrac{\partial \varphi_1}{\partial x}(X^{(u_1,u_2)}(T),\alpha (T) )X_1(T)\Big ]   \label{eqIprim2} \\
=&E\Big[ p_1(T)X(T)\Big ]- E\Big[\dfrac{\partial h_1}{\partial x}(X^{(u_1,u_2)}(T),\alpha(T) )A_1(T)X_1(T)\Big  ] \notag\\
=&E\Big[\int_0^T\Big\{p_1(t)\Big(\dfrac{\partial b}{\partial x}(t)X_1(t)+\dfrac{\partial b}{\partial u_1}(t)\beta_1(t)\Big)-X_1(t) \dfrac{\partial H_1}{\partial x}(t)\Big.\Big.\notag\\
& +q_1(t)\Big(\dfrac{\partial \sigma}{\partial x}(t)X_1(t)+\dfrac{\partial \sigma}{\partial u_1}(t)\beta_1(t)\Big)+\int_{\mathbb{R}_0} r_1(t,\zeta)\Big(\dfrac{\partial \gamma}{\partial x}(t,\zeta)X_1(t) +\dfrac{\partial \gamma}{\partial u_1}(t,\zeta)\beta_1(t)\Big) \nu_\alpha(\diffns \zeta)\bigg.\notag\\
& +\sum_{j=1}^D w_1^j(t)\Big(\dfrac{\partial \eta^j}{\partial x}(t)X_1(t)-\dfrac{\partial \eta^j}{\partial u_1}(t)\beta_1(t) \Big)\lambda_{j}(t) \Big\}\diffns t\Big]\notag\\
&- E\Big[\dfrac{\partial h_1}{\partial x}(X^{(u_1,u_2)}(T),\alpha(T) )A_1(T)X_1(T)\Big ]\Big ].\notag
\end{align}
Applying once more the It\^o's formula and using \eqref{derivassBSDE1} and \eqref{thneccond12}, we get

\begin{align}
J_3=&E\Big[\psi_1^\prime(Y(0))y_1(0)\Big] =E\Big[A(0)y_1(0)\Big]\notag\\
=&E\Big[A_1(T)y_1(T)\Big] - E\Big[\int_0^T\Big\{A_1(t^-)\diff y_1(t)+y_1(t^-)\diff A_1(t) +\dfrac{\partial H_1}{\partial z}(t)z_1(t)\diff t\Big.\Big.\notag\\
&\Big.\Big.+\int_{\mathbb{R}_0}\nabla_kH_1(t,\zeta)k_1(t,\zeta)\nu_\alpha(\diffns  \zeta)\diff t+\sum_{j=1}^D\dfrac{\partial H_1}{\partial v_1^j}(t)  v^j_1(t)\lambda_{j}(t) \diff t  \Big\}\Big]\notag
\end{align}
\begin{align}
=&E\Big[\dfrac{\partial h_1}{\partial x}(X^{(u_1,u_2)}(T),\alpha(T))X_1(T)\Big] + E\Big[\int_0^T\Big\{A_1(t)\Big(\dfrac{\partial g_1}{\partial x}(t)x_1(t)  +\Big.\dfrac{\partial g_1}{\partial y}(t)y_1(t) \Big.\Big.\Big. \notag\\
&  +\dfrac{\partial g_1}{\partial z}(t)z_1(t)+\int_{\mathbb{R}_0}\nabla_k g_1 (t,\zeta)k_1(t,\zeta)\nu_\alpha(\diffns \zeta)   + \sum_{j=1}^D\dfrac{\partial g_1}{\partial v^j}(t)  v^j_1(t)\lambda_{j}(t)    \notag\\
&   +\dfrac{\partial g_1}{\partial u_1}(t)\beta_1(t)   \Big)-\dfrac{\partial H_1}{\partial y}(t)y_1(t) -\dfrac{\partial H_1}{\partial z}(t)z_1(t)-\int_{\mathbb{R}_0}\nabla_kH_1(t,\zeta)k_1(t,\zeta)\nu_\alpha(\diffns \zeta) \notag\\
& -\sum_{j=1}^D\dfrac{\partial H_1}{\partial v^j}(t)  v^j_1(t)\lambda_{j}(t) \Big\}\diffns  t\Big] .\label{eqIprim3}
\end{align}

Substituting \eqref{eqIprim2} and \eqref{eqIprim3} into \eqref{eqIprim11}, we get

\begin{align}\label{eqIprim311}
&\dfrac{\diffns }{\diffns \ell}J_1^{(u_1+\ell \beta_1,u_2)}(t)\Big. \Big|_{\ell=0}\notag\\
=&E\Big[\int_0^TX_1(t)  \Big\{\dfrac{\partial f_1}{\partial x}(t) +A_1(t)\dfrac{\partial g_1}{\partial x}(t)+p_1(t)\dfrac{\partial b}{\partial x}(t)+q_1(t)\dfrac{\partial \sigma}{\partial x}(t)+\int_{\mathbb{R}_0} r_1(t,\zeta)\dfrac{\partial \gamma}{\partial x}(t,\zeta)\nu_\alpha(\diffns  \zeta)\notag\\
&+\sum_{j=1}^Dw_1^j(t)\dfrac{\partial \eta^j}{\partial x}(t)\lambda_{j}(t) -\dfrac{\partial H_1}{\partial x}(t)\Big\}\diffns t\Big]+E\Big[\int_0^T y_1(t)  \Big\{\dfrac{\partial f_1}{\partial y}(t) +A_1(t)\dfrac{\partial g_1}{\partial y}(t)-\dfrac{\partial H_1}{\partial y}(t)\Big\}\diffns t\Big]  \notag\\
&+E\Big[\int_0^T z_1(t)  \Big\{\dfrac{\partial f_1}{\partial z}(t) +A_1(t)\dfrac{\partial g_1}{\partial z}(t)-\dfrac{\partial H_1}{\partial z}(t)\Big\}\diffns t\Big] \notag\\
&+E\Big[\int_0^T\int_{\mathbb{R}_0}k_1(t,\zeta) \Big\{\nabla_k f_1(t,\zeta)+A_1(t)\nabla_kg_1(t,\zeta)-\nabla_k H_1(t,\zeta)\Big\}\nu_\alpha(\diffns \zeta) \diffns t\notag\\
&+E\Big[\int_0^T \sum_{j=1}^D v_1^j(t)  \Big\{\dfrac{\partial f_1}{\partial v^j}(t) +A_1(t)\dfrac{\partial g}{\partial v^j}(t)-\dfrac{\partial H}{\partial v^j}(t)\Big\}\diffns t\Big] \notag \\
&+E\Big[\int_0^T \beta_1(t)  \Big\{\dfrac{\partial f_1}{\partial u_1}(t) +A_1(t)\dfrac{\partial g_1}{\partial u_1}(t)+\dfrac{\partial b}{\partial u_1}(t)+\dfrac{\partial \sigma}{\partial u_1}(t)\notag\\
&+\int_{\mathbb{R}_0} r_1(t,\zeta)\dfrac{\partial \gamma}{\partial u_1}(t,\zeta)\nu_\alpha(\diffns  \zeta)
+\sum_{j=1}^D w_1^j(t)\dfrac{\partial \eta^j}{\partial u_1}(t)\lambda_{j}(t) \Big\}\diffns t\Big].
\end{align}
By the definition of $H_1$, the coefficients of the processes $X_1(t),y_1(t),z_1(t), k_1(t,\zeta)$ and $v_1^j(t),\, j=1,\ldots,D,$ are all equal to zero in \eqref{eqIprim311}. We conclude that
\begin{align}
\dfrac{\diffns }{\diffns \ell}J_1^{(u_1+\ell \beta_1,u_2)}(t)\Big. \Big|_{\ell=0}=&E\Big[\int_0^T\dfrac{\partial H_1}{\partial u_1}(t)\beta_1(t) \diff t \Big]\notag\\
=& E\Big[\int_0^TE\Big[\dfrac{\partial H_1}{\partial u_1}(t)\beta_1(t) \Big| \mathcal{E}^{(1)}_{t} \Big]\diff t \Big].
\end{align}
Hence,
$\dfrac{\diffns }{\diffns \ell}J_1^{(u_1+\ell \beta_1,u_2)}(t)\Big. \Big|_{\ell=0}=0$ for all bounded $ \beta_1 \in \mathcal{A}_1$ implies that the same holds in particular for $\beta_1 \in \mathcal{A}_1$ of the form
$$
\beta_1(t)=\beta_1(t,\omega)=\theta_1(\omega)\xi_{[t_0,T]}(t), t\in [0,T]
$$
for a fix $t_0\in [0,T)$, where $\theta_1(\omega)$ is a bounded $\mathcal{E}^{(1)}_{t_0}$-measurable random variable. Therefore
\begin{align}
E\Big[\int_{t_0}^TE\Big[\dfrac{\partial H_1}{\partial u_1}(t) \Big| \mathcal{E}^{(1)}_{t} \Big]\theta_1\diff t \Big]=0.
\end{align}
Differentiating with respect to $t_0$, we have
\begin{align}
E\Big[\dfrac{\partial H_1}{\partial u_1}(t_0)\,\theta_1\Big]=0 \text{ for a.a., } t_0.
\end{align}
Since the equality is true for all bounded $\mathcal{E}^{(1)}_{t_0}$-measurable random variables $\theta_1$, we have
\begin{align}
E\Big[\dfrac{\partial H_1}{\partial u_1}(t_0)|\mathcal{E}^{(1)}_{t_0}\Big]=0 \text{ for a.a., } t_0\in[0,T].
\end{align}
 A similar argument gives that
 $$
E\Big[\dfrac{\partial H_2}{\partial u_2}(t_0)|\mathcal{E}^{(2)}_{t_0}\Big]=0 \text{ for a.a., } t_0\in[0,T],
$$
 under the condition that
 $$\dfrac{\diffns }{\diffns s}J^{(u_1,u_2+\ell \beta_2)}(t)\Big. \Big|_{\ell=0}=0\,\, \text{ for all bounded } \beta_2 \in \mathcal{A}_2.$$

This shows that (1) $\Rightarrow$ (2).

Conversely, using the fact that every bounded $\beta_i \in \mathcal{A}_i$  can be approximated by a linear combinations of controls $\beta_i(t)$ of the form \eqref{eqbeta1}, the above argument can be reversed to show that  (2) $\Rightarrow$ (1).
\end{proof}

\bibliographystyle{plain}
\bibliography{Biblio}

\end{document}